\documentclass[11pt]{article}
    \usepackage{url}
    \usepackage{verbatim}
    \usepackage[titletoc]{appendix}
    \usepackage{graphicx}
	\usepackage{amsmath}
	\usepackage{amsthm}
	\usepackage{amsfonts}
	\usepackage{graphicx}
    \usepackage{nicefrac}
    \usepackage{longtable}
    \usepackage{color}
    \usepackage{graphicx, amssymb, subcaption,graphics}

\newtheorem{rem}{Remark}[section]
\newtheorem{thm}{Theorem}[section]

\newtheorem{lem}{Lemma}[section]

\newcommand{\nabh}{\nabla_{\! h}}

\newcommand{\hf}{\nicefrac{1}{2}}
\newcommand{\nrm}[1]{\left\| #1 \right\|}
\newcommand{\ciptwo}[2]{\left\langle #1 , #2 \right\rangle}

\newcommand\dt {{\Delta t}}

\newcommand{\eipx}[2]{\left[ #1 , #2 \right]_{\rm x}}
\newcommand{\eipy}[2]{\left[ #1 , #2 \right]_{\rm y}}
\newcommand{\eipz}[2]{\left[ #1 , #2 \right]_{\rm z}}
\newcommand{\eipvec}[2]{\left[ #1 , #2 \right]}

\newcommand{\reff}[1]{{\rm (\ref{#1})}}
	\newcommand\be {\begin{equation}}
	\newcommand\ee {\end{equation}}

\newcommand{\hh}{\mbox{\boldmath$h$}}
	\title{A positivity-preserving, energy stable and convergent numerical scheme for the Poisson-Nernst-Planck system}

	\date{\today}

\begin{document}
	
	\author{
Chun Liu \thanks{Department of Applied Mathematics, Illinois Institute of Technology, IL 60616, USA (cliu124@iit.edu)}
\and	
Cheng Wang\thanks{Department of Mathematics, University of Massachusetts, North Dartmouth, MA  02747, USA (cwang1@umassd.edu)}
	\and
Steven M. Wise\thanks{Department of Mathematics, The University of Tennessee, Knoxville, TN 37996, USA (swise1@utk.edu)} 
\and	
Xingye Yue\thanks{Department of Mathematics, Soochow University, Suzhou 215006, P.R. China (xyyue@suda.edu.cn)}
\and	
Shenggao Zhou\thanks{Department of Mathematics, Soochow University, Suzhou 215006, P.R. China (Corresponding Author: sgzhou@suda.edu.cn)}
}

 	\maketitle
	\numberwithin{equation}{section}
	
\begin{abstract}	 
In this paper we propose and analyze a finite difference numerical scheme for the Poisson-Nernst-Planck equation (PNP) system. To understand the energy structure of the PNP model, we make use of the Energetic Variational Approach (EnVarA), so that the PNP system could be reformulated as a non-constant mobility $H^{-1}$ gradient flow, with singular logarithmic energy potentials involved. To ensure the unique solvability and energy stability, the mobility function is explicitly treated, while both the logarithmic and the electric potential diffusion terms are treated implicitly, due to the convex nature of these two energy functional parts. The positivity-preserving property for both concentrations, $n$ and $p$, is established at a theoretical level. This is based on the subtle fact that the singular nature of the logarithmic term around the value of $0$ prevents the numerical solution reaching the singular value, so that the numerical scheme is always well-defined. In addition, an optimal rate convergence analysis is provided in this work, in which many highly non-standard estimates have to be involved, due to the nonlinear parabolic coefficients. The higher order asymptotic expansion (up to third order temporal accuracy and fourth order spatial accuracy),  the rough error estimate (to establish the $\ell^\infty$ bound for $n$ and $p$), and the refined error estimate have to be carried out to accomplish such a convergence result. In our knowledge, this work will be the first to combine the following three theoretical properties for a numerical scheme for the PNP system: (i) unique solvability and positivity, (ii) energy stability, and (iii) optimal rate convergence. A few numerical results are also presented in this article, which demonstrates the robustness of the proposed numerical scheme.

	\bigskip

\noindent
{\bf Key words and phrases}:
Poisson-Nernst-Planck (PNP) system, logarithmic energy potential, positivity preserving, energy stability, optimal rate convergence analysis, higher order asymptotic expansion 

\noindent
{\bf AMS subject classification}: \, 35K35, 35K55, 65M12, 65M06, 82C70	
\end{abstract}
	

\section{Introduction} 

We consider the two-particle Poisson-Nernst-Planck (PNP) system of equations  
	\begin{eqnarray} 
\partial_t n &=&  D_n \Delta n - \frac{z_0 e_0}{k_B \theta_0} \nabla \cdot \left( D_n n \nabla \phi  \right)   ,   
	\label{equation-PNP-1} 
	\\
\partial_t p &=&   D_p \Delta p + \frac{z_0 e_0}{k_B \theta_0} \nabla \cdot \left( D_p p \nabla \phi  \right)   ,   
	\label{equation-PNP-2}  
	\\
 - \varepsilon \Delta \phi &=& z_0 e_0 (p-n) + \rho^f , 
 	\label{equation-PNP-3} 
	\end{eqnarray}
where $k_B$ is the Boltzmann constant; $\theta_0$ is the absolute temperature;  $n$ and $p$ are the concentrations of negatively and positively charged ions, respectively; $\varepsilon$ is the dielectric coefficient of the solution; $z_0$ is valence of ions; $e_0$ is the charge an electron; $\phi$ is the electric potential; and $D_n$ and $D_p$ are diffusion/mobility coefficients. Boundary conditions are very important for PNP systems and must be handled carefully~\cite{Flavell14}. However, we will assume periodic boundary conditions in this work for simplicity of presentation. The analysis could be extended to more complicated, more physical boundary conditions. In addition, for simplicity of presentation in the theoretical analysis, we assume that source term, $\rho^f$, associated to the background fixed charge density, vanishes everywhere. The extension to a non-zero source term is straightforward.  

The PNP system is one of the most extensively studied models for the transport of charged particles in many physical and biological problems, including free electrons in semiconductors~\cite{jerome95, Markowich86, Markowich90}; fuel cells~\cite{Nazarov07, promislow01}; ionic particles in electrokinetic fluids~\cite{Ben02, Hunter01, Lyklema95}; phase separation and polarization for ionic liquids~\cite{Gavish16}; and ion channels in cell membranes~\cite{Bazant04, Eisenberg96, Nonner99}. The Energetic Variational Approach (EnVarA)~\cite{Eisenberg10} shows that the PNP system is the gradient flow with respect to a particular free energy. In more detail, the free energy functional of a two-particle mixture may be formulated as 
	\begin{equation}
E (n, p) = \int_\Omega\left\{ k_B \theta_0 \left( n \ln \frac{n}{n_0} + p \ln \frac{p}{p_0} \right) \right\}d{\bf x} + \frac{z_0^2 e_0^2}{2 \varepsilon} \| n - p \|_{H^{-1}}^2 , 
	\label{PNP energy}
	\end{equation} 
under the assumption that $n-p$ is of mean zero, where $n_0$ and $p_0$ are reference concentrations. The $H^{-1}$ norm is defined via
	\[
\nrm{f}_{H^{-1}} : = 	\sqrt{\left(f,f\right)_{H^{-1}}},
	\]
where 
	\[
\left(f,g\right)_{H^{-1}} := \left(\nabla\psi_f,\nabla\psi_g \right)_{L^2},
	\]
and $\psi_f\in \mathring{H}^1_{\rm per}(\Omega):= H^1_{\rm per}(\Omega)\cap \mathring{L}^2(\Omega)$ is the solution to
	\[
-\Delta \psi_f	= f\in \mathring{L}^2(\Omega) := \left\{ f\in L^2(\Omega) \, \middle| \, (f,1)_{L^2} = 0 \right\}.
	\]
Formally, then
	\[
\nrm{f}_{H^{-1}}^2  = \left( f, (-\Delta)^{-1} f \right)_{L^2} .	
	\]
The PNP system~\eqref{equation-PNP-1} -- \eqref{equation-PNP-3} is the following $H^{-1}$-like gradient flow: 
	\begin{equation} 
\partial_t n = \nabla \cdot \left(\frac{D_n}{k_B \theta_0} n \nabla \mu_n \right),  \quad  \partial_t p = \nabla \cdot \left( \frac{D_p}{k_B \theta_0} p \nabla \mu_p \right) , 
 \  \label{equation-PNP-0} 
	\end{equation}
where $\mu_n$ and $\mu_p$ are chemical potentials given by
	\begin{align} 
   \mu_n & := \delta_n E = k_B \theta_0 ( \ln \frac{n}{n_0} + 1 ) + \frac{z_0^2 e_0^2}{\varepsilon}  (-\Delta)^{-1} ( n - p)    = k_B \theta_0 ( \ln \frac{n}{n_0} + 1 ) - z_0 e_0 \phi ,        
	\label{PNP-chem pot-n} 
	\\
\mu_p &:= \delta_p E = k_B \theta_0 ( \ln \frac{p}{p_0} + 1 ) + \frac{z_0^2 e_0^2}{\varepsilon}  (-\Delta)^{-1} ( p - n)  = k_B \theta_0 ( \ln \frac{p}{p_0} + 1 ) + z_0 e_0 \phi ,       
	\label{PNP-chem pot-p}
	\end{align} 
and $\phi$ is the periodic and mean-zero solution to 
	\[
 - \varepsilon \Delta \phi = z_0 e_0 (p-n).
	\]
Of course, for the system to make sense, we require that the initial data satisfy
	\[
\frac{1}{|\Omega|}\int_\Omega n({\bf x},0)\, d{\bf x} = \frac{1}{|\Omega|}\int_\Omega p({\bf x},0)\, d{\bf x} > 0.
	\]
Notice that non-constant coefficient mobility functions are involved in the formulated gradient flow. 

It is clear that the PDE solutions are conserved, positive (in the sense that $n,p>0$, point-wise) and energy dissipative. There are a number of papers describing numerical methods for the PNP system. However, the theoretical analysis for numerical approximations turns out to be very challenging, in particular for those based on the EnVarA formulation. First, the positivity of $n$ and $p$ have to be enforced to make the numerical scheme well-defined in the EnVarA formulation. Some existing works have reported a positivity-preserving analysis~\cite{DingWangZhou_JCP2019, DingWangZhou_JCP2020, Flavell14, HuHuang_NM20, LiuH14, LiuMing_Sub20, LiuMing_CICP20, SiddWangZhou_CMS18}, while many of these analyses come from the maximum principle argument, instead of in the variational framework. Second, the energy stability has also played a central role in the study of gradient flows. Such a stability analysis has appeared in a few existing numerical works~\cite{Flavell17b, LW17, XuLiu_JCP16}, while the unique solvability and positivity-preserving analysis have been missing. Furthermore, there have been a few existing works for the convergence analysis~\cite{ding19, ProhlSchmuck_NM09, SunY16}, while these convergence estimates have been based on the perfect Laplacian operator structure for $n$ and $p$, instead of the $H^{-1}$ gradient flow structure, so that the energy estimate is not available.  Many other numerical schemes have been reported~\cite{GaoHe_JSC17, HeD15, XuLiu_JCP16, Gibou_JCP14, QianWangZhou_JCP20, TuB13, XuS14b}. However, no existing work has combined the following three theoretical features in the numerical analyses: (i) unique solvability/positivity preserving property, (ii) energy stability in the variational framework, and (iii) optimal rate convergence analysis.
 
In this paper we construct and analyze a finite difference numerical scheme, which preserves all three important theoretical features. For the energy stability property, the numerical scheme has to be based on the variational structure of the original PNP system. The mobility function is explicitly updated in the scheme to enforce the strictly elliptic nature of the operator associated with the temporal derivative part in the $H^{-1}$ gradient flow. For the chemical potential part, all the terms are treated implicitly, because of the convex nature of both the logarithmic and the electric potential diffusion energy parts (in terms of $n$ and $p$). Moreover, the positivity-preserving property, for both $n$ and $p$, will be theoretically established. Such an analysis is based on the fact that the numerical solution is equivalent to the minimization of the numerical energy functional, and the singular nature of the logarithmic term around the value of $0$ prevents the numerical solution reaching a singular value. As a result, the numerical scheme is always well-defined, and the unique solvability analysis results from the convex nature of the implicit parts in the scheme. Such a technique has been successfully applied to the Cahn-Hilliard model~\cite{chen19b, dong19b, dong20a}, while its application to the PNP system will be involved more subtle details, due to the non-constant mobility. Furthermore, the energy stability comes directly from the corresponding convexity analysis, combined with the positivity of the mobility functions. 
  
We provide an optimal rate convergence analysis for the proposed numerical scheme. The variational structure and the non-constant mobility make this analysis highly challenging, especially when compared with existing convergence estimates in~\cite{ding19, ProhlSchmuck_NM09, SunY16}, wherein a perfect Laplacian operator is kept in tact. To overcome such a well-known difficulty, several highly non-standard estimates have to be introduced, due to the nonlinear parabolic coefficients. The higher order asymptotic expansion, up to the third order temporal accuracy and fourth order spatial accuracy, has to be performed with a careful linearization technique. Such a higher order asymptotic expansion enables one to obtain a rough error estimate, so that to the $\ell^\infty$ bound for $n$ and $p$ could be derived. This $\ell^\infty$ estimate yields the upper and lower bounds of the two variables, and these bounds play a crucial role in the subsequent analysis. Finally, the refined error estimate is carried out to accomplish the desired convergence result. To our knowledge, it will be the first work to combine three theoretical properties for any numerical scheme for the PNP system: unique solvability/positivity-preserving, energy stability, and optimal rate convergence analysis. 
  
The rest of the article is organized as follows. In Section~\ref{sec:numerical scheme} we propose the fully discrete numerical scheme. The detailed proof for the positivity-preserving property of the numerical solution is provided in Section~\ref{sec:positivity}, and the energy stability analysis is established in Section~\ref{sec:energy stability}. The optimal rate convergence analysis is presented in Section~\ref{sec:convergence}.  Some numerical results are provided in Section~\ref{sec:numerical results}. Finally, the concluding remarks are given in Section~\ref{sec:conclusion}.

	\section{The fully discrete numerical scheme}
	\label{sec:numerical scheme}	

	\subsection{Nondimensionalization}
	\label{subsec:non-dim}
	
We introduce the dimensionless dependent variables $\hat{n} := n/n_0$, $\hat{p} := p/p_0$, with $c_0 = n_0 = p_0$, and $\hat{\phi} := \phi/\phi_0$, with
	\[
\phi_0 = \frac{k_B\theta_0}{z_0e_0}.
	\]
We use the dimensionless independent variables $\hat{x} := x/L$ and $\hat{t} := t/T$, with
	\[
L = \sqrt{\frac{\varepsilon k_B\theta_0}{(z_0e_0)^2 c_0}} \quad \mbox{and} \quad T = \frac{L^2}{D_n}.
	\]
Define $\hat{D} := D_p/D_n$. Then the dimensionless dynamical equations may be written (after dropping the hats on the parameters and variables) as
		\begin{eqnarray} 
\partial_t n &=&   \nabla\cdot\left( \nabla n - n \nabla \phi  \right)   ,   
	\label{equation-PNP-1-nd} 
	\\
\partial_t p &=&   D \nabla\cdot\left( \nabla p +  p \nabla \phi  \right)   ,   
	\label{equation-PNP-2-nd}  
	\\
 -  \Delta \phi &=&  p-n  .
 	\label{equation-PNP-3-nd} 
	\end{eqnarray}
This system dissipates the dimensionless energy
	\begin{equation}
E (n, p) = \int_\Omega\left\{   n \ln n + p \ln p +  \frac{1}{2} (n-p)(-\Delta)^{-1}(n-p)\right\}d{\bf x} ,
	\label{PNP energy-nd}
	\end{equation} 	
and may be viewed as the following conserved gradient flow:
	\begin{equation} 
\partial_t n  = \nabla \cdot \left( n \nabla \mu_n \right),  \quad  
\partial_t p = D\nabla \cdot \left(  p \nabla \mu_p \right) , 
 \  \label{equation-PNP-0-nd} 
	\end{equation}
where $\mu_n$ and $\mu_p$ are the dimensionless chemical potentials given by
	\begin{align} 
\mu_n & := \delta_n E = \ln n + 1 + (-\Delta)^{-1} ( n - p)   =  \ln  n + 1  -  \phi ,        
	\label{PNP-chem pot-n-nd} 
	\\
\mu_p &:= \delta_p E =   \ln p +1 + (-\Delta)^{-1} ( p - n)  =   \ln p +1 +  \phi ,       
	\label{PNP-chem pot-p-nd}
	\end{align} 
and $\phi$ is the periodic solution to 
	\[
 - \Delta \phi = p-n.
	\]
Consequently, the energy is dissipated at the rate
	\[
d_t E = -\int_\Omega\left\{n \, \left|\nabla\mu_n \right|^2+ D\, p \,  \left|\nabla\mu_p \right|^2 \right\} d{\bf x} \le 0.
	\]

	\subsection{The finite difference spatial discretization}
	\label{subsec:finite difference}

We use the notation and results for some discrete functions and operators from~\cite{guo16, wise10, wise09a}. Let $\Omega = (-L_x, L_x)\times(-L_y, L_y)\times(-L_z, L_z)$, where for simplicity, we assume $L_x =L_y=L_z =: L > 0$. Let $N\in\mathbb{N}$ be given, and define the grid spacing $h := \frac{2L}{N}$, i.e., a uniform spatial mesh size is taken for simplicity of presentation.  We define the following two uniform, infinite grids with grid spacing $h>0$: $E := \{ p_{i+\hf} \ |\ i\in {\mathbb{Z}}\}$, $C := \{ p_i \ |\ i\in {\mathbb{Z}}\}$, where $p_i = p(i) := (i-\hf)\cdot h$. Consider the following 3-D discrete $N^3$-periodic function spaces: 
	\begin{eqnarray*}
	\begin{aligned}
{\mathcal C}_{\rm per} &:= \left\{\nu: C\times C
\times C\rightarrow {\mathbb{R}}\ \middle| \ \nu_{i,j,k} = \nu_{i+\alpha N,j+\beta N, k+\gamma N}, \ \forall \, i,j,k,\alpha,\beta,\gamma\in \mathbb{Z} \right\},
	\\
{\mathcal E}^{\rm x}_{\rm per} &:=\left\{\nu: E\times C\times C\rightarrow {\mathbb{R}}\ \middle| \ \nu_{i+\frac12,j,k}= \nu_{i+\frac12+\alpha N,j+\beta N, k+\gamma N}, \ \forall \, i,j,k,\alpha,\beta,\gamma\in \mathbb{Z}\right\} , 
	\end{aligned}
	\end{eqnarray*}
in which identification $\nu_{i,j,k} = \nu(p_i,p_j,p_k)$ is taken. The spaces  ${\mathcal E}^{\rm y}_{\rm per}$ and ${\mathcal E}^{\rm z}_{\rm per}$ are analogously defined. The functions of ${\mathcal C}_{\rm per}$ are called {\emph{cell centered functions}}. The functions of ${\mathcal E}^{\rm x}_{\rm per}$, ${\mathcal E}^{\rm y}_{\rm per}$, and ${\mathcal E}^{\rm z}_{\rm per}$,  are called {\emph{east-west}},  {\emph{north-south}}, and {\emph{up-down face-centered functions}}, respectively.  We also define the mean zero space 
	\[
\mathring{\mathcal C}_{\rm per}:= \left\{\nu\in {\mathcal C}_{\rm per} \ \middle| \ 0 = \overline{\nu} :=  \frac{h^3}{| \Omega|} \sum_{i,j,k=1}^N \nu_{i,j,k} \right\},
	\] 
and denote $\vec{\mathcal{E}}_{\rm per} := {\mathcal E}^{\rm x}_{\rm per}\times {\mathcal E}^{\rm y}_{\rm per}\times {\mathcal E}^{\rm z}_{\rm per}$. In addition, we introduce the important difference and average operators on the spaces:  
	\begin{eqnarray*}
&& A_x \nu_{i+\hf,j,k} := \frac{1}{2}\left(\nu_{i+1,j,k} + \nu_{i,j,k} \right), \quad D_x \nu_{i+\hf,j,k} := \frac{1}{h}\left(\nu_{i+1,j,k} - \nu_{i,j,k} \right),
	\\
&& A_y \nu_{i,j+\hf,k} := \frac{1}{2}\left(\nu_{i,j+1,k} + \nu_{i,j,k} \right), \quad D_y \nu_{i,j+\hf,k} := \frac{1}{h}\left(\nu_{i,j+1,k} - \nu_{i,j,k} \right) , 
	\\
&& A_z \nu_{i,j,k+\hf} := \frac{1}{2}\left(\nu_{i,j,k+1} + \nu_{i,j,k} \right), \quad D_z \nu_{i,j,k+\hf} := \frac{1}{h}\left(\nu_{i,j,k+1} - \nu_{i,j,k} \right) , 
	\end{eqnarray*}
with $A_x,\, D_x: {\mathcal C}_{\rm per}\rightarrow{\mathcal E}_{\rm per}^{\rm x}$, $A_y,\, D_y: {\mathcal C}_{\rm per}\rightarrow{\mathcal E}_{\rm per}^{\rm y}$, $A_z,\, D_z: {\mathcal C}_{\rm per}\rightarrow{\mathcal E}_{\rm per}^{\rm z}$. 
Likewise,
	\begin{eqnarray*}
&& a_x \nu_{i, j, k} := \frac{1}{2}\left(\nu_{i+\hf, j, k} + \nu_{i-\hf, j, k} \right),	 \quad d_x \nu_{i, j, k} := \frac{1}{h}\left(\nu_{i+\hf, j, k} - \nu_{i-\hf, j, k} \right),
	\\
&& a_y \nu_{i,j, k} := \frac{1}{2}\left(\nu_{i,j+\hf, k} + \nu_{i,j-\hf, k} \right),	 \quad d_y \nu_{i,j, k} := \frac{1}{h}\left(\nu_{i,j+\hf, k} - \nu_{i,j-\hf, k} \right),
	\\
&& a_z \nu_{i,j,k} := \frac{1}{2}\left(\nu_{i, j,k+\hf} + \nu_{i, j, k-\hf} \right),	 \quad d_z \nu_{i,j, k} := \frac{1}{h}\left(\nu_{i, j,k+\hf} - \nu_{i, j,k-\hf} \right),
	\end{eqnarray*}
with $a_x,\, d_x : {\mathcal E}_{\rm per}^{\rm x}\rightarrow{\mathcal C}_{\rm per}$, $a_y,\, d_y : {\mathcal E}_{\rm per}^{\rm y}\rightarrow{\mathcal C}_{\rm per}$, and $a_z,\, d_z : {\mathcal E}_{\rm per}^{\rm z}\rightarrow{\mathcal C}_{\rm per}$. 
The discrete gradient $\nabh:{\mathcal C}_{\rm per}\rightarrow \vec{\mathcal{E}}_{\rm per}$ and the discrete divergence $\nabh\cdot :\vec{\mathcal{E}}_{\rm per} \rightarrow {\mathcal C}_{\rm per}$ are given by 
	\[
\nabh\nu_{i,j,k} =\left( D_x\nu_{i+\hf, j, k},  D_y\nu_{i, j+\hf, k},D_z\nu_{i, j, k+\hf}\right) , 
  \quad 
\nabh\cdot\vec{f}_{i,j,k} = d_x f^x_{i,j,k}	+ d_y f^y_{i,j,k} + d_z f^z_{i,j,k},
	\]
where $\vec{f} = (f^x,f^y,f^z)\in \vec{\mathcal{E}}_{\rm per}$. The standard 3-D discrete Laplacian, $\Delta_h : {\mathcal C}_{\rm per}\rightarrow{\mathcal C}_{\rm per}$, becomes  
	\begin{align*}
\Delta_h \nu_{i,j,k} := & \nabla_{h}\cdot\left(\nabla_{h}\nu\right)_{i,j,k} =  d_x(D_x \nu)_{i,j,k} + d_y(D_y \nu)_{i,j,k}+d_z(D_z \nu)_{i,j,k}
	\\
= & \ \frac{1}{h^2}\left( \nu_{i+1,j,k}+\nu_{i-1,j,k}+\nu_{i,j+1,k}+\nu_{i,j-1,k}+\nu_{i,j,k+1}+\nu_{i,j,k-1} - 6\nu_{i,j,k}\right).
	\end{align*}
More generally, if $\mathcal{D}$ is a periodic \emph{scalar} function that is defined at all of the face center points and $\vec{f}\in\vec{\mathcal{E}}_{\rm per}$, then $\mathcal{D}\vec{f}\in\vec{\mathcal{E}}_{\rm per}$, assuming point-wise multiplication, and we may define
	\[
\nabla_h\cdot \big(\mathcal{D} \vec{f} \big)_{i,j,k} = d_x\left(\mathcal{D}f^x\right)_{i,j,k}  + d_y\left(\mathcal{D}f^y\right)_{i,j,k} + d_z\left(\mathcal{D}f^z\right)_{i,j,k} .
	\]
Specifically, if $\nu\in \mathcal{C}_{\rm per}$, then $\nabla_h \cdot\left(\mathcal{D} \nabla_h  \cdot \right):\mathcal{C}_{\rm per} \rightarrow \mathcal{C}_{\rm per}$ is defined point-wise via
	\[
\nabla_h\cdot \big(\mathcal{D} \nabla_h \nu \big)_{i,j,k} = d_x\left(\mathcal{D}D_x\nu\right)_{i,j,k}  + d_y\left(\mathcal{D} D_y\nu\right)_{i,j,k} + d_z\left(\mathcal{D}D_z\nu\right)_{i,j,k} .
	\]

In addition, the following grid inner products are defined:   
	\begin{equation*}
	\begin{aligned}
\ciptwo{\nu}{\xi}  &:= h^3\sum_{i,j,k=1}^N  \nu_{i,j,k}\, \xi_{i,j,k},\quad \nu,\, \xi\in {\mathcal C}_{\rm per},\quad
& \eipx{\nu}{\xi} := \langle a_x(\nu\xi) , 1 \rangle ,\quad \nu,\, \xi\in{\mathcal E}^{\rm x}_{\rm per},
\\
\eipy{\nu}{\xi} &:= \langle a_y(\nu\xi) , 1 \rangle ,\quad \nu,\, \xi\in{\mathcal E}^{\rm y}_{\rm per},\quad
&\eipz{\nu}{\xi} := \langle a_z(\nu\xi) , 1 \rangle ,\quad \nu,\, \xi\in{\mathcal E}^{\rm z}_{\rm per}.
	\end{aligned}
	\end{equation*}	
	\[
[ \vec{f}_1 , \vec{f}_2 ] : = \eipx{f_1^x}{f_2^x}	+ \eipy{f_1^y}{f_2^y} + \eipz{f_1^z}{f_2^z}, \quad \vec{f}_i = (f_i^x,f_i^y,f_i^z) \in \vec{\mathcal{E}}_{\rm per}, \ i = 1,2.
	\]
Subsequently, we define the following norms for cell-centered functions. If $\nu\in {\mathcal C}_{\rm per}$, then $\nrm{\nu}_2^2 := \langle \nu , \nu \rangle$; $\nrm{\nu}_p^p := \ciptwo{|\nu|^p}{1}$, for $1\le p< \infty$, and $\nrm{\nu}_\infty := \max_{1\le i,j,k\le N}\left|\nu_{i,j,k}\right|$. The gradient norms are introduced as follows: 
	\begin{eqnarray*} 
\nrm{ \nabla_h \nu}_2^2 &: =& \eipvec{\nabh \nu }{ \nabh \nu } = \eipx{D_x\nu}{D_x\nu} + \eipy{D_y\nu}{D_y\nu} +\eipz{D_z\nu}{D_z\nu},  \quad \forall \,  \nu \in{\mathcal C}_{\rm per} ,   
	\\
\nrm{\nabla_h \nu}_p^p &:=&  \eipx{|D_x\nu|^p}{1} + \eipy{|D_y\nu|^p}{1} +\eipz{|D_z\nu|^p}{1}   , \quad \forall \, \nu \in{\mathcal C}_{\rm per}, \quad  1\le p<\infty . 
	\end{eqnarray*}
Higher order norms can be defined. For example,
	\[
\nrm{\nu}_{H_h^1}^2 : =  \nrm{\nu}_2^2+ \nrm{ \nabla_h \nu}_2^2, \quad \nrm{\nu}_{H_h^2}^2 : =  \nrm{\nu}_{H_h^1}^2  + \nrm{ \Delta_h \nu}_2^2  , \quad \forall \, \nu \in{\mathcal C}_{\rm per}.
	\]

	\begin{lem}[\cite{wang11a, wise09a}]
	\label{lemma1}    
Let $\mathcal{D}$ be an arbitrary periodic, scalar function defined on all of the face center points. For any $\psi, \nu \in {\mathcal C}_{\rm per}$ and any $\vec{f}\in\vec{\mathcal{E}}_{\rm per}$, the following summation-by-parts formulas are valid: 
	\begin{equation}
\langle \psi , \nabla_h\cdot\vec{f} \rangle = - [ \nabla_h \psi , \vec{f} ], \quad 
  \langle \psi, \nabla_h\cdot \left(\mathcal{D}\nabla_h\nu\right) \rangle 
  = - [ \nabla_h \psi , \mathcal{D}\nabla_h\nu ] .
	\label{lemma 1-0} 
	\end{equation}
	\end{lem}

	\subsection{The numerical scheme} 
	
For simplicity, we denote  $( {\cal M}^m_n )_{i,j,k}=   n_{i,j,k}^m$, $( {\cal M}^m_p )_{i,j,k}=  D p^m_{i,j,k}$, and introduce the following mobility function at the face-centered mesh points: 
	\begin{align} 
( \breve{\cal M}_n^m )_{i+\hf,j,k} &:= A_x ({\cal M}_n^m )_{i+\hf, j,k} ,
	\nonumber
	\\
( \breve{\cal M}_n^m )_{i,j+\hf,k} &:= A_y ({\cal M}_n^m )_{i, j+\hf ,k} ,
	\label{mob ave-1} 
	\\
( \breve{\cal M}_n^m )_{i,j,k+\hf} &:= A_z ( {\cal M}_n^m )_{i, j,k+\hf} , 
	\nonumber
	\end{align} 
with similar definitions for $\breve{\cal M}_p^m$. We use the following semi-implicit scheme: given $n^m,p^m\in {\mathcal C}_{\rm per}$, find  $n^{m+1},p^{m+1}\in {\mathcal C}_{\rm per}$ such that
	\begin{align} 
\frac{n^{m+1} - n^m}{\dt} & = \nabla_h \cdot \left( \breve{\cal M}_n^m \nabla_h \mu_n^{m+1}  \right) , 
 	\label{scheme-PNP-1} 
	\\
\frac{p^{m+1} - p^m}{\dt} & = \nabla_h \cdot \left( \breve{\cal M}_p^m \nabla_h \mu_p^{m+1}  \right) , 
 	\label{scheme-PNP-2}	
	\\
\mu_n^{m+1} & =  \ln n^{m+1}+ (-\Delta_h)^{-1} ( n^{m+1} - p^{m+1} )  ,
	\label{scheme-PNP-chem pot-n} 
	\\
\mu_p^{m+1} & =  \ln p^{m+1} + (-\Delta_h)^{-1} ( p^{m+1} - n^{m+1} )  .   
	\label{scheme-PNP-chem pot-p}   
	\end{align}

	\section{Positivity-preserving and unique solvability analyses} 
	\label{sec:positivity} 

Recall the average operator: $\overline{f} = \frac{1}{| \Omega|} \langle f , {\bf 1} \rangle$. It is obvious that the numerical scheme~\eqref{scheme-PNP-1} -- \eqref{scheme-PNP-chem pot-p} is mass conservative, so that 
	\[
\overline{n^m} = \overline{n^0} := \beta_0, \quad \overline{p^m} = \overline{p^0} := \beta_0, \quad \mbox{with} \quad 0 < \beta_0, \quad  \forall \, m \ge 1.
	\]
The following preliminary estimates, which are proved in the recent paper~\cite{chen19b}, are recalled.  For any $\varphi  \in\mathring{\mathcal C}_{\rm per}$, there exists a unique $\psi\in\mathring{\mathcal C}_{\rm per}$ that solves
	\begin{equation}
\mathcal{L}_{\breve{\cal M} } (\psi):= - \nabla_h \cdot ( \check{\cal M} \nabla_h \psi ) 
 = \varphi .
	\label{PNP-mobility-0} 
	\end{equation}
The following discrete norm may be defined: 
	\begin{equation} 
\| \varphi  \|_{\mathcal{L}_{\breve{\cal M} }^{-1} } = \sqrt{ \langle \varphi ,  
 \mathcal{L}_{\check{\cal M} }^{-1} (\varphi) \rangle } .   
	\end{equation} 
If $\check{\cal M}\equiv 1$, we  have $\mathcal{L}_{\breve{\cal M} } (\psi) = -\Delta_h\psi$ and define
	\begin{equation} 
\| \varphi  \|_{-1,h} = \sqrt{ \langle \varphi ,  
 (-\Delta_h )^{-1} (\varphi) \rangle } .   
	\end{equation}

	\begin{lem}[\cite{chen19b}]
	\label{PNP-positivity-Lem-0}
Suppose that $\varphi^\star$, $\hat{\varphi} \in \mathcal{C}_{\rm per}$, with $\hat{\varphi} - \varphi^\star\in\mathring{\mathcal{C}}_{\rm per}$. Assume that $0 < \hat{\varphi}_{i,j,k} , \varphi^\star_{i,j,k} \le M_h$, for all $1 \le i,j,k \le N$, where $M_h >0$ may depend on $h$. The following estimate is valid: 
	\begin{equation} 
\|  ( - \Delta_h)^{-1} ( \hat{\varphi} - \varphi^\star ) \|_\infty \le \tilde{C}_1 M_h ,
  	\label{PNP-Lem-0} 
	\end{equation} 
where $\tilde{C}_1>0$ only depends on $\Omega$. 
	\end{lem} 

	\begin{lem}[\cite{chen19b}]
	\label{PNP-mobility-positivity-Lem-0}  
Suppose that $\varphi_1$, $\varphi_2 \in \mathcal{C}_{\rm per}$, with  $\varphi_1 - \varphi_2\in \mathring{\mathcal{C}}_{\rm per}$. Assume that $\| \varphi_1 \|_\infty , \| \varphi_2\|_\infty \le M_h$, and $\breve{\cal M} \ge {\cal M}_0$ at a point-wise level, for some constant ${\cal M}_0>0$ that is independent of $h$. Then we have the following estimate: 
	\begin{equation} 
\nrm{ \mathcal{L}_{\breve{\cal M} }^{-1} (\varphi_1 - \varphi_2)}_\infty \le C_2 := \tilde{C}_2 \mathcal{M}_0^{-1} h^{-1/2} ,  
	\label{PNP-mobility-Lem-0} 
	\end{equation} 
where $\tilde{C}_2>0$ depends only upon $M_h$ and $\Omega$.
	\end{lem} 

The positivity-preserving and unique solvability properties are established in the following theorem.

	\begin{thm}  
	\label{PNP-positivity} 
Given $n^m,p^m\in  {\mathcal{C}}_{\rm per}$, with $0 <  n^m_{i,j,k}, p^m_{i,j,k}$, $1 \le i, j, k \le N$,  and $n^m-p^m\in\mathring{\mathcal{C}}_{\rm per}$, there exists a unique solution $(n^{m+1},p^{m+1})\in \left[{\mathcal{C}}_{\rm per}\right]^2$ to the numerical scheme~\eqref{scheme-PNP-1} -- \eqref{scheme-PNP-chem pot-p}, with $0 < n^{m+1}_{i,j,k}, p^{m+1}_{i,j,k}$, $1 \le i, j, k \le N$ and $n^{m+1}-p^{m+1}\in\mathring{\mathcal{C}}_{\rm per}$.    
\end{thm}

	\begin{proof} 
Suppose, as before, that $\overline{n^m} = \overline{p^m} =\beta_0 >0$. Define $\nu^m := n^m-\beta_0$ and $\rho^m := p^m-\beta_0$. The numerical solution of~\eqref{scheme-PNP-1} -- \eqref{scheme-PNP-chem pot-p} is equivalent to the minimization of the following discrete energy functional: 
	\begin{align} 
J^m_h (\nu, \rho) &= \frac{1}{2 \dt} \left( \| \nu - \nu^m \|_{\mathcal{L}_{\breve{\cal M}_n^m }^{-1} }^2 + \| \rho - \rho^m \|_{\mathcal{L}_{\breve{\cal M}_p^m }^{-1} }^2  \right) 
	\nonumber
	\\
& \quad +\langle (\nu+\beta_0) \ln (\nu+\beta_0) + (\rho+\beta_0) \ln (\rho+\beta_0)  , {\bf 1} \rangle  + \frac{1}{2} \| \nu - \rho \|_{-1,h}^2  ,
	\label{PNP-positive-1} 
	\end{align} 
over the admissible set
	\begin{equation}
\mathring{A}_h :=  \left\{ (\nu, \rho) \in \left[\mathring{\mathcal C}_{\rm per}\right]^2 \ \middle| \ 0 < \nu_{i,j,k}+\beta_0 , \, \rho_{i,j,k}+\beta_0 < M_h, \ 1\le i, j, k\le N  \right\} ,
	\end{equation} 
where $M_h := \nicefrac{( \beta_0 | \Omega | ) }{h^3}$. 
We observe that $J^m_h (n, p)$ is a strictly convex function over this domain.  Next, we prove that there exists a minimizer of $J^m_h (n, p)$ over the domain $\mathring{A}_h$.

Consider the following closed domain: for $\delta >0$, 
	\begin{equation}
\mathring{A}_{h,\delta} :=  \left\{ (\nu, \rho) \in \left[\mathring{\mathcal C}_{\rm per}\right]^2 \ \middle| \ \delta \le \nu_{i,j,k}+\beta_0 , \, \rho_{i,j,k}+\beta_0 \le M_h-\delta, \ 1\le i, j, k\le N  \right\}  .
	\end{equation} 
Since $\mathring{A}_{h,\delta}$ is a compact set in the hyperplane $H:=\left\{ (\nu, \rho )\middle| \overline{\nu} = \overline{\rho} = 0 \right\}$, there exists a (not necessarily unique) minimizer of $J^m_h (\nu, \rho)$ over $\mathring{A}_{h,\delta}$. The key point of the positivity analysis is that, such a minimizer could not occur at one of the boundary points (in $H$) if $\delta$ is sufficiently small. 

Let us suppose that the minimizer of $J^m_h (\nu,\rho)$ occurs at a boundary point of $\mathring{A}_{h,\delta}$. Without loss of generality, we assume the minimizer is $( \nu^\star_{i,j,k} , \rho^\star_{i,j,k} )$, with $\nu^\star_{{i_0},{j_0},{k_0}}+\beta_0=\delta$, at some grid point $({i_0}, {j_0}, {k_0})$. Suppose that $\nu^\star$ attains its maximum value at the point $({i_1}, {j_1}, {k_1})$. By the fact that $\overline{\nu^\star} = 0$, it is obvious that $\nu^\star_{{i_1}, {j_1}, {k_1}} \ge 0$. 

Consider the following directional derivative: for any $\psi\in\mathring{\mathcal C}_{\rm per}$, 
	\begin{align*}
d_s \! \left.J^m_h(\nu^\star +s\psi ,\rho^\star)\right|_{s=0} & = \frac{1}{\Delta t} \ciptwo{\mathcal{L}_{\breve{\cal M}_n^m }^{-1}\left(\nu^\star-\nu^m\right)}{\psi} + \ciptwo{\ln\left(\nu^\star+\beta_0\right)+1}{\psi}
	\\
& \quad + \ciptwo{(-\Delta_h)^{-1}\left(\nu^\star -\rho^\star \right)}{\psi}.
	\end{align*}
Let us pick the direction $\psi \in \mathring{\mathcal{C}}_{\rm per}$, such that  
	\[
\psi_{i,j,k} = \delta_{i,i_0}\delta_{j,j_0}\delta_{k,k_0} - \delta_{i,i_1}\delta_{j,j_1}\delta_{k,k_1} ,
	\] 
where $\delta_{k,\ell}$ is the Kronecker delta function. Then,
	\begin{align}
\frac{1}{h^3}d_s \! \left.J^m_h(\nu^\star +s\psi ,\rho^\star)\right|_{s=0} & =   \ln \left(\frac{\nu^\star_{{i_0},{j_0},{k_0}}+\beta_0}{\nu^\star_{{i_1},{j_1},{k_1}} +\beta_0}\right)   
	\nonumber 
	\\
& \quad+ (-\Delta_h)^{-1} ( \nu^\star - \rho^\star )_{{i_0},{j_0},{k_0}} - (-\Delta_h)^{-1} ( \nu^\star - \rho^\star )_{{i_1},{j_1},{k_1}} 
	\nonumber 
	\\
& \quad + \frac{1}{\dt} \left( \mathcal{L}_{\breve{\cal M}_n^m }^{-1}  ( \nu^\star - \nu^m )_{{i_0},{j_0},{k_0}} - \mathcal{L}_{\breve{\cal M}_n^m }^{-1}  ( \nu^\star - \nu^m )_{{i_1},{j_1},{k_1}} \right)  .
	\label{PNP-positive-4} 
	\end{align}
Because
	\[
n^\star_{{i_0}, {j_0}, {k_0}} = \nu^\star_{{i_0}, {j_0}, {k_0}}+\beta_0 = \delta \quad \mbox{and} \quad n^\star_{{i_1},{j_1},{k_1}} = \nu^\star_{{i_1},{j_1},{k_1}} +\beta_0 \ge \beta_0,
	\]
we have 
	\begin{equation} 
\ln \left(\frac{\nu^\star_{{i_0},{j_0},{k_0}}+\beta_0}{\nu^\star_{{i_1},{j_1},{k_1}} +\beta_0}\right)   \le \ln \frac{\delta}{\beta_0}.
	\label{PNP-positive-5} 
	\end{equation} 
For the third and fourth terms appearing in~\eqref{PNP-positive-4}, we apply Lemma~\ref{PNP-positivity-Lem-0} and obtain 
	\begin{equation} 
- 2 \tilde{C}_1 M_h \le (-\Delta_h)^{-1} ( \nu^\star -\rho^\star )_{{i_0},{j_0},{k_0}} - (-\Delta_h)^{-1} ( \nu^\star - \rho^\star )_{{i_1},{j_1},{k_1}} \le  2 \tilde{C}_1 M_h .
	\label{PNP-positive-6} 
	\end{equation}
Similarly, for the last two terms appearing in~\eqref{PNP-positive-4}, an application of Lemma~\ref{PNP-mobility-positivity-Lem-0} indicates that 
	\begin{equation} 
- 2 \tilde{C}_2 \mathcal{M}_0^{-1} h^{-1/2} \le \mathcal{L}_{\breve{\cal M}_n^m }^{-1}  ( \nu^\star - \nu^m )_{{i_0},{j_0},{k_0}} - \mathcal{L}_{\breve{\cal M}_n^m }^{-1}  ( \nu^\star - \nu^m )_{{i_1},{j_1},{k_1}} \le  2 \tilde{C}_2 \mathcal{M}_0^{-1} h^{-1/2} .  
	\label{PNP-positive-7} 
	\end{equation}
Consequently,  a substitution of~\eqref{PNP-positive-5} -- \eqref{PNP-positive-7} into~\eqref{PNP-positive-4} yields 
	\begin{equation} 
\frac{1}{h^3}d_s \! \left.J^m_h(\nu^\star +s\psi ,\rho^\star)\right|_{s=0} \le \ln \frac{\delta}{\beta_0} +  2\tilde{C}_1 M_h + 2 \tilde{C}_2 \mathcal{M}_0^{-1} \dt^{-1} h^{-1/2}.  
	\label{PNP-positive-8} 
	\end{equation}  

Define 
	\[
D_0 := 2\tilde{C}_1 M_h + 2 \tilde{C}_2 \mathcal{M}_0^{-1}\dt^{-1} h^{-1/2},
	\]
and note that $D_0$ is a constant for fixed $\dt$ and $h$, though it is singular, as $\dt, h\to 0$. For any fixed $\dt$ and $h$, we may choose $\delta>0$ small enough so that 
	\begin{equation} 
\ln \frac{\delta}{\beta_0}  + D_0 < 0 .  
	\label{PNP-positive-9} 
	\end{equation} 
This in turn gaurantees that 
	\begin{equation} 
d_s \! \left.J^m_h(\nu^\star +s\psi ,\rho^\star)\right|_{s=0}  < 0 .  
	\label{PNP-positive-10} 
	\end{equation} 
This contradicts the assumption that $J^m_h$ has a minimum at $(\nu^\star ,\rho^\star)$, since the directional derivative is negative in a direction pointing into the interior of $\mathring{A}_{h,\delta}$.

Using similar arguments, we can also prove that, the global minimum of $J^m_h (\nu,\rho)$ over $\mathring{A}_{h,\delta}$ could not possibly occur at a boundary point satisfying $\rho^\star_{{i_0},{j_0},{k_0}} +\beta_0 =\delta$, if $\delta$ is small enough. The details are left to interested readers. 
  
Therefore, the global minimum of $J^m_h (\nu,\rho)$ over $\mathring{A}_{h,\delta}$ could only possibly occur at an interior point, for $\delta>0$ sufficiently small. Since $J^m_h  (\nu,\rho)$ is a smooth function, we conclude that there must be a solution $(\nu_{i,j,k}, \rho_{i,j,k}) \in \mathring{A}_h$, so that 
	\begin{equation} 
d_s \! \left.J^m_h(\nu +s\psi ,\rho +s\phi)\right|_{s=0} =0 ,  \quad \forall \, (\psi,\phi) \in \left[\mathring{\mathcal C}_{\rm per}\right]^2,
	\label{PNP-positive-13} 
	\end{equation} 
which is equivalent to the numerical solution of~\eqref{scheme-PNP-1} -- \eqref{scheme-PNP-chem pot-p}. The existence of a positive numerical solution is established. 

Finally, since $J^m_h (\nu, \rho)$ is a strictly convex function over $\mathring{A}_h$, the uniqueness analysis for this numerical solution is straightforward. The proof of Theorem~\ref{PNP-positivity} is complete. 
	\end{proof}

	\section{Energy stability analysis}
	\label{sec:energy stability}

With the positivity-preserving and unique solvability properties for the numerical scheme~\eqref{scheme-PNP-1} -- \eqref{scheme-PNP-chem pot-p} established, we now prove energy stability. We introduce the following discrete energy: 
	\begin{eqnarray} 
E_h (n, p) := \langle n \ln n + p \ln p , {\bf 1} \rangle + \frac{1}{2} \| n - p \|_{-1,h}^2 . 
	\label{PNP-discrete energy}
	\end{eqnarray}

	\begin{thm}
	\label{PNP-energy stability} 
For the numerical solution~\eqref{scheme-PNP-1} -- \eqref{scheme-PNP-chem pot-p}, we have 
	\begin{equation} 
E_h (n^{m+1}, p^{m+1}) + \dt \left(\eipvec{\check{\mathcal{M}}_n^{m} \nabla_h \mu_n^{m+1}}{\nabla_h \mu_n^{m+1}} + \eipvec{\check{\mathcal{M}}_p^{m} \nabla_h \mu_p^{m+1}}{\nabla_h \mu_p^{m+1} } \right)  
   \le E_h (n^{m}, p^{m}),   \label{PNP-energy-0} 
	\end{equation} 
so that $E_h (n^m, p^m) \le E_h (n^0, p^0) \le C_0$, for all $m\in \mathbb{N}$, where $C_0>0$ is a constant independent of $h$. 
	\end{thm}
	\begin{proof} 
Taking discrete inner products of~\eqref{scheme-PNP-1} with $\mu_n^{m+1}$ and of~\eqref{scheme-PNP-2} with $\mu_p^{m+1}$, we obtain
	\begin{align}
& \ciptwo{n^{m+1} - n^m}{\mu_n^{m+1}} + \ciptwo{p^{m+1} - p^m}{\mu_p^{m+1}} 
	\nonumber
	\\
&\qquad + \dt \left( \eipvec{\check{\mathcal{M}}_n^m \nabla_h \mu_n^{m+1}}{\nabla_h \mu_n^{m+1}}  + \eipvec{\check{\mathcal{M}}_p^m \nabla_h \mu_p^{m+1}}{\nabla_h \mu_p^{m+1}} \right) = 0 .
	\label{PNP-energy-1} 
	\end{align}
On the other hand, the convexity of the energy terms $\langle n \ln n , {\bf 1}  \rangle$, $\langle p \ln p , {\bf 1}  \rangle$ and $\| n - p \|_{-1,h}^2$ imply that 
	\begin{align} 
\ciptwo{n^{m+1} - n^m}{\ln n^{m+1}} & \ge \ciptwo{n^{m+1} \ln n^{m+1}}{{\bf 1}} - \ciptwo{n^m \ln n^m}{ {\bf 1}} ,
	\label{PNP-energy-2-1}
	\\
\ciptwo{p^{m+1} - p^m}{\ln p^{m+1}} & \ge \ciptwo{p^{m+1} \ln p^{m+1}}{{\bf 1}} - \ciptwo{p^m \ln p^m}{ {\bf 1}} , 
     \label{PNP-energy-2-2}
	\\
\ciptwo{n^{m+1} - n^m}{(-\Delta_h)^{-1} ( n^{m+1} - p^{m+1} )} &+ \ciptwo{p^{m+1} - p^m}{ (-\Delta_h)^{-1} ( p^{m+1} - n^{m+1} )}  
	\nonumber
	\\ 
 & \ge \frac12 \left(  \nrm{n^{m+1} - p^{m+1}}_{-1,h}^2 - \nrm{ n^m - p^m }_{-1,h}^2 \right) .
	\label{PNP-energy-2-3}
	\end{align}  
Substitution of~\eqref{PNP-energy-2-1} -- \eqref{PNP-energy-2-3} into \eqref{PNP-energy-1} leads to~\eqref{PNP-energy-0}, so that the unconditional energy stability is proved. 

Finally, that there is a constant $C_0>0$ that independent of $h$, such that $E_h (n^0, p^0) \le C_0$, follows from a consistency argument. The details are left to the interested reader.
	\end{proof}

	\section{Optimal rate convergence analysis in $\ell^\infty (0,T; \ell^2) \cap \ell^2 (0,T; H_h^1)$}
	\label{sec:convergence}

Now we proceed into the convergence analysis. Let $({\mathsf N}, {\mathsf P}, \Phi)$ be the exact PDE solution for the non-dimensional PNP system~\eqref{equation-PNP-1-nd} -- \eqref{equation-PNP-3-nd}. With sufficiently regular initial data, it is reasonable to assume that the exact solution has regularity of class $\mathcal{R}$, where
	\begin{equation}
{\mathsf N}, {\mathsf P}  \in \mathcal{R} := H^4 \left(0,T; C_{\rm per}(\Omega)\right) \cap H^3 \left(0,T; C^2_{\rm per}(\Omega)\right) \cap L^\infty \left(0,T; C^6_{\rm per}(\Omega)\right).
	\label{assumption:regularity.1}
	\end{equation}
In addition, we assume that the following separation property is valid for the exact solution:
	\begin{equation} 
{\mathsf N}  \ge \epsilon_0 ,\quad  {\mathsf P} \ge \epsilon_0 ,  \quad \mbox{for some $\epsilon_0 > 0$} , 
    \label{assumption:separation}
	\end{equation}  
which we assume holds at a point-wise level. Define ${\mathsf N}_N (\, \cdot \, ,t) := {\cal P}_N {\mathsf N} (\, \cdot \, ,t)$, ${\mathsf P}_N (\, \cdot \, ,t) := {\cal P}_N {\mathsf P} (\, \cdot \, ,t)$, the (spatial) Fourier projection of the exact solution into ${\cal B}^K$, the space of trigonometric polynomials of degree to and including  $K$ (with $N=2K +1$).  The following projection approximation is standard: if $({\mathsf N}, {\mathsf P}) \in L^\infty(0,T;H^\ell_{\rm per}(\Omega))$, for any $\ell\in\mathbb{N}$ with $0 \le k \le \ell$, 
	\begin{equation}
\nrm{ {\mathsf N}_N - {\mathsf N}}_{L^\infty(0,T;H^k)} \le C h^{\ell-k} \| {\mathsf N} \|_{L^\infty(0,T;H^\ell)},  \ \nrm{ {\mathsf P}_N - {\mathsf P} }_{L^\infty(0,T;H^k)} \le C h^{\ell-k} \| {\mathsf P} \|_{L^\infty(0,T;H^\ell)} . \label{projection-est-0}
	\end{equation}
Notice that the Fourier projection estimate~\eqref{projection-est-0} does not preserve the positivity of the variables, while we could take $h$ sufficiently small (corresponding to a large $N$) so that ${\mathsf N}_N \ge \frac12 \epsilon_0$, ${\mathsf P}_N \ge \frac12 \epsilon_0$. 
	
	
By ${\mathsf N}_N^m$, ${\mathsf P}_N^m$ we denote $ {\mathsf N}_N (\, \cdot \, , t_m)$ and ${\mathsf P}_N (\, \cdot \, , t_m)$, respectively, with $t_m = m\cdot \dt$. Since $({\mathsf N}_N, {\mathsf P}_N) \in {\cal B}^K$, the mass conservative property is available at the discrete level:
	\begin{eqnarray} 
	  &&
\overline{ {\mathsf N}_N^m} = \frac{1}{|\Omega|}\int_\Omega \, {\mathsf N}_N ( \cdot, t_m) \, d {\bf x} = \frac{1}{|\Omega|}\int_\Omega \, {\mathsf N}_N ( \cdot, t_{m-1}) \, d {\bf x} = \overline{ {\mathsf N}_N^{m-1}} ,  \label{mass conserv-1-1}
\\
   	 &&
\overline{ {\mathsf P}_N^m} = \frac{1}{|\Omega|}\int_\Omega \, {\mathsf P}_N ( \cdot, t_m) \, d {\bf x} = \frac{1}{|\Omega|}\int_\Omega \, {\mathsf P}_N ( \cdot, t_{m-1}) \, d {\bf x} = \overline{ {\mathsf P}_N^{m-1}} ,  \label{mass conserv-1-2}
	\end{eqnarray}
for any $m \in \mathbb{N}$. On the other hand, the solution of~\eqref{scheme-PNP-1} -- \eqref{scheme-PNP-2} is also mass conservative at the discrete level:
	\begin{equation}
\overline{n^m} = \overline{n^{m-1}} ,  \, \, \, 
\overline{p^m} = \overline{p^{m-1}} , \quad \forall \ m \in \mathbb{N} . 
	\label{mass conserv-2}
	\end{equation}
As indicated before, we use the mass conservative projection for the initial data:  $n^0 = {\mathcal P}_h {\mathsf N}_N (\, \cdot \, , t=0)$, $p^0 = {\mathcal P}_h {\mathsf P}_N (\, \cdot \, , t=0)$, that is
	\begin{equation}
(n^0)_{i,j,k} := {\mathsf N}_N (p_i, p_j, p_k, t=0) , \quad (p^0)_{i,j,k} := {\mathsf P}_N (p_i, p_j, p_k, t=0) . 
	\label{initial data-0}
	\end{equation}	
For the exact electric potential $\Phi$, we denote its Fourier projection as $\Phi_N$. The error grid function is defined as
	\begin{equation}
e_n^m := \mathcal{P}_h {\mathsf N}_N^m - n^m ,  \, \, \, 
e_p^m := \mathcal{P}_h {\mathsf P}_N^m - p^m , \, \, \, 
e_\phi^m := \mathcal{P}_h \Phi_N^m - \phi^m , 
\quad \forall \ m \in \mathbb{N} .
	\label{error function-1}
	\end{equation}
Therefore, it follows that  $\overline{e_n^m} = \overline{e_p^m} =0$, for any $m \in \mathbb{N}$,  so that the discrete norm $\nrm{ \, \cdot \, }_{-1,h}$ is well defined for the error grid function.

The following theorem is the main result of this section.

\begin{thm}
	\label{thm:convergence}
Given initial data ${\mathsf N}(\, \cdot \, ,t=0), {\mathsf P} (\, \cdot \, ,t=0) \in C^6_{\rm per}(\Omega)$, suppose the exact solution for the PNP system~\eqref{equation-PNP-1-nd} -- \eqref{equation-PNP-2-nd} is of regularity class $\mathcal{R}$. Then, provided $\dt$ and $h$ are sufficiently small, and under the linear refinement requirement $C_1 h \le \dt \le C_2 h$, we have
	\begin{equation}
\| e_n^m \|_2 + \| e_p^m \|_2 + \Bigl( \dt   \sum_{k=1}^{m} ( \| \nabla_h e_n^k \|_2^2 
+ \| \nabla_h e_p^k \|_2^2 ) \Bigr)^{1/2}  + \| e_\phi^m \|_{H_h^2} \le C ( \dt + h^2 ) , 
	\label{convergence-0}
	\end{equation}
for all positive integers $m$, such that $t_m=m\dt \le T$, where $C>0$ is independent of $\dt$ and $h$.
	\end{thm}
	
	
\subsection{Higher order consistency analysis of~\eqref{scheme-PNP-1} -- \eqref{scheme-PNP-chem pot-p}:  asymptotic expansion of the numerical solution}
\label{subsec-consistency}

By consistency, the project solution ${\mathsf N}_N$, ${\mathsf P}_N$ solves the discrete equations~\eqref{scheme-PNP-1} -- \eqref{scheme-PNP-chem pot-p} with a first order accuracy in time and second order accuracy in space. Meanwhile, it is observed that this leading local truncation error will not be enough to recover an a-priori $\ell^\infty$ bound for the numerical solution to recover the separation property. To remedy this, we use a higher order consistency analysis, via a perturbation argument, to recover such a bound in later analysis. In more detail, we need to construct supplementary fields, ${\mathsf N}_{\dt, 1}$, ${\mathsf N}_{\dt, 2}$, ${\mathsf P}_{\dt ,1}$, ${\mathsf P}_{\dt, 2}$ and $\check{\mathsf N}$, $\check{\mathsf P}$ satisfying
	\begin{equation}
 \check{\mathsf N} = {\mathsf N}_N + {\cal P}_N ( \dt {\mathsf N}_{\dt, 1} + \dt^2 {\mathsf N}_{\dt, 2} 
  + h^2 {\mathsf N}_{h, 1} ), \, \, \,  
 \check{\mathsf P} = {\mathsf P}_N + {\cal P}_N ( \dt {\mathsf P}_{\dt, 1} + \dt^2 {\mathsf P}_{\dt, 2} + h^2 {\mathsf P}_{h, 1} ) , 
	\label{consistency-1}
	\end{equation} 
so that a higher $O (\dt^3 + h^4)$ consistency is satisfied with the given numerical scheme~\eqref{scheme-PNP-1} -- \eqref{scheme-PNP-chem pot-p}.  The constructed fields ${\mathsf N}_{\dt ,j}$, ${\mathsf N}_{h, 1}$, ${\mathsf P}_{\dt, j}$, ${\mathsf P}_{h, 1}$which will be found using a perturbation expansion, will depend solely on the exact solution $({\mathsf N}, P)$.

The following truncation error analysis for the temporal discretization can be obtained by using a straightforward Taylor expansion, as well as the estimate~\eqref{projection-est-0} for the projection solution:  	
\begin{eqnarray} 
 \frac{{\mathsf N}_N^{m+1} - {\mathsf N}_N^m}{\dt} &=&  \nabla \cdot \Big( 
    {\mathsf N}_N^m \nabla ( 
    \ln  {\mathsf N}_N^{m+1}  +    
  (-\Delta)^{-1} ( {\mathsf N}_N^{m+1} - {\mathsf P}_N^{m+1} )  ) \Big)   \nonumber 
\\
  &&
   + \dt ( G_n^{(0)} )^m + O(\dt^2) + O (h^{m_0}) , 
     \label{consistency-2-1} 
\\ 
  \frac{ {\mathsf P}_N^{m+1} - {\mathsf P}_N^m}{\dt} &=&  \nabla \cdot \Big( 
  D {\mathsf P}_N^m \nabla ( 
    \ln  {\mathsf P}_N^{m+1} +    
  (-\Delta)^{-1} ( {\mathsf P}_N^{m+1} - {\mathsf N}_N^{m+1} )  )   \Big)   \nonumber 
\\
  &&
   + \dt ( G_p^{(0)} )^m + O(\dt^2) + O (h^{m_0}) .  
     \label{consistency-2-2} 
 \end{eqnarray}  
Here $m_0 \ge 4$ and the spatial functions $G_n^{(0)}$ $G_p^{(0)}$ are smooth enough in the sense that their derivatives are bounded. 

The leading order temporal correction function $({\mathsf N}_{\dt, 1}, {\mathsf P}_{\dt, 1})$ is given by solving the following equations:
\begin{eqnarray} 
  \partial_t {\mathsf N}_{\dt, 1}  &=& 
  \nabla \cdot \Big( 
    {\mathsf N}_{\dt, 1} \nabla ( 
    \ln  {\mathsf N}_N  +    
  (-\Delta)^{-1} ( {\mathsf N}_N - {\mathsf P}_N )  )    \nonumber 
\\
  &&  \quad 
   +   {\mathsf N}_N \nabla (   
    \frac{1}{{\mathsf N}_N} {\mathsf N}_{\dt, 1} +    
  (-\Delta)^{-1} ( {\mathsf N}_{\dt, 1} - {\mathsf P}_{\dt, 1} )  ) \Big)   
   - G_n^{(0)}  , 
     \label{consistency-3-1} 
\\
    \partial_t {\mathsf P}_{\dt, 1} &=& 
  \nabla \cdot \Big( 
  D {\mathsf P}_{\dt, 1} \nabla ( 
    \ln  {\mathsf P}_N  +    
  (-\Delta)^{-1} ( {\mathsf P}_N - {\mathsf N}_N )  )    \nonumber 
\\
  &&  \quad 
   + D {\mathsf P}_N \nabla (   
    \frac{1}{{\mathsf P}_N} {\mathsf P}_{\dt, 1} +    
  (-\Delta)^{-1} ( {\mathsf P}_{\dt, 1} - {\mathsf N}_{\dt, 1}  )  ) \Big)   
   - G_p^{(0)}  .  
     \label{consistency-3-2}      
\end{eqnarray}
Existence of a solution of the above linear PDE system is straightforward. Note that the solution depends only on the projection solution $({\mathsf N}_N, {\mathsf P}_N)$. In addition, the derivatives of $({\mathsf N}_{\dt, 1}, {\mathsf P}_{\dt, 1})$ in various orders are bounded. Of course, an application of the semi-implicit discretization (as given by~\eqref{consistency-2-1} -- \eqref{consistency-2-2}) to~\eqref{consistency-3-1} -- \eqref{consistency-3-2} implies that 
\begin{eqnarray} 
 \frac{ {\mathsf N}_{\dt, 1}^{m+1} - {\mathsf N}_{\dt, 1}^m}{\dt} &=& 
  \nabla \cdot \Big( 
    {\mathsf N}_{\dt, 1}^m \nabla ( 
    \ln  {\mathsf N}_N^{m+1}   +    
  (-\Delta)^{-1} ( {\mathsf N}_N^{m+1} - {\mathsf P}_N^{m+1} )  )    \nonumber 
\\
  &&  \, \,  
   +   {\mathsf N}_N^m \nabla (   
    \frac{1}{{\mathsf N}_N^{m+1} }  {\mathsf N}_{\dt, 1}^{m+1} 
     +    
  (-\Delta)^{-1} (  {\mathsf N}_{\dt, 1}^{m+1} - {\mathsf P}_{\dt, 1}^{m+1} )  ) \Big)    
  \nonumber 
\\
  &&  \quad 
   - ( G_n^{(0)} )^m  + \dt \hh_1^n + O (\dt^2) ,  \label{consistency-4-1} 
\\
   \frac{ {\mathsf P}_{\dt, 1}^{m+1} -  {\mathsf P}_{\dt, 1}^m}{\dt} &=& 
  \nabla \cdot \Big( 
  D {\mathsf P}_{\dt, 1}^m \nabla ( 
    \ln  {\mathsf P}_N^{m+1}  +    
  (-\Delta)^{-1} ( {\mathsf P}_N^{m+1} - {\mathsf N}_N^{m+1} )  )    \nonumber 
\\
  &&  \, \,  
   + D ( {\mathsf P}_N )^m \nabla (   
    \frac{1}{{\mathsf P}_N^{m+1} }  {\mathsf P}_{\dt, 1}^{m+1}  
    +    
  (-\Delta)^{-1} ( {\mathsf P}_{\dt, 1}^{m+1} - {\mathsf N}_{\dt, 1}^{m+1} )  ) \Big)   
  \nonumber 
\\
  &&  \quad 
   - ( G_p^{(0)} )^m + \dt \hh_2^n + O (\dt^2) .  
     \label{consistency-4-2}      
\end{eqnarray}
Therefore, a combination of~\eqref{consistency-2-1} -- \eqref{consistency-2-2} and \eqref{consistency-4-1} -- \eqref{consistency-4-2} leads to the second order temporal truncation error for $\check{\mathsf N}_1 := {\mathsf N}_N + \dt {\cal P}_N {\mathsf N}_{\dt, 1}$, $\check{\mathsf P}_1 := {\mathsf P}_N + \dt {\cal P}_N {\mathsf P}_{\dt, 1}$: 
\begin{eqnarray} 
  \frac{\check{\mathsf N}_1^{m+1} - \check{\mathsf N}_1^m}{\dt} &=& 
  \nabla \cdot \Big( 
    \check{\mathsf N}_1^m \nabla ( 
    \ln \check{\mathsf N}_1^{m+1}  +    
  (-\Delta)^{-1} ( \check{\mathsf N}_1^{m+1} - \check{\mathsf P}_1^{m+1} )  ) \Big)   \nonumber 
\\
  &&
   + \dt^2 ( G_n^{(1)} )^m + O(\dt^3) + O (h^{m_0}) , 
     \label{consistency-5-1} 
\\ 
   \frac{\check{\mathsf P}_1^{m+1} - \check{\mathsf P}_1^m}{\dt} &=& 
    \nabla \cdot \Big( 
   D \check{\mathsf P}_1^m \nabla ( 
     \ln \check{\mathsf P}_1^{m+1} +    
  (-\Delta)^{-1} ( \check{\mathsf P}_1^{m+1} - \check{\mathsf N}_1^{m+1} )  )   \Big)   \nonumber 
\\
  &&
   + \dt^2 ( G_p^{(1)} )^m + O(\dt^3) + O (h^{m_0}) .  
     \label{consistency-5-2} 
\end{eqnarray} 
In the derivation of~\eqref{consistency-5-1} -- \eqref{consistency-5-2}, the following linearized expansions have been utilized: 
\begin{eqnarray} 
	  &&
	\ln \check{\mathsf N}_1 = \ln ( {\mathsf N}_N + \dt {\cal P}_N {\mathsf N}_{\dt, 1} ) 
	= \ln {\mathsf N}_N + \frac{\dt {\cal P}_N {\mathsf N}_{\dt, 1}}{{\mathsf N}_N} + O (\dt^2) , 
	\label{consistency-6-1}	
\\
  &&
    \ln \check{\mathsf P}_1 = \ln ( {\mathsf P}_N + \dt {\cal P}_N {\mathsf P}_{\dt, 1} ) 
	= \ln {\mathsf P}_N + \frac{\dt {\cal P}_N {\mathsf P}_{\dt, 1}}{{\mathsf P}_N} + O (\dt^2) .   
	\label{consistency-6-2}       	
\end{eqnarray} 

Similarly, the next order temporal correction function $({\mathsf N}_{\dt, 2}, {\mathsf P}_{\dt, 2})$ are given by the following linear equations: 
\begin{eqnarray} 
  \partial_t {\mathsf N}_{\dt, 2}  &=& 
  \nabla \cdot \Big( 
    {\mathsf N}_{\dt, 2} \nabla ( 
    \ln  \check{\mathsf N}_1 +    
  (-\Delta)^{-1} ( \check{\mathsf N}_1 - \check{\mathsf P}_1 )  )    \nonumber 
\\
  &&  \quad 
   +   \check{\mathsf N}_1 \nabla (   
    \frac{1}{\check{\mathsf N}_1} {\mathsf N}_{\dt, 2} +    
  (-\Delta)^{-1} ( {\mathsf N}_{\dt, 2} - {\mathsf P}_{\dt, 2} )  ) \Big)   
   - G_n^{(1)}  , 
     \label{consistency-7-1} 
\\
    \partial_t {\mathsf P}_{\dt, 2} &=& 
  \nabla \cdot \Big( 
  D {\mathsf P}_{\dt, 2} \nabla ( 
    \ln \check{\mathsf P}_1 +    
  (-\Delta)^{-1} ( \check{\mathsf P}_1 - \check{\mathsf N}_1 )  )    \nonumber 
\\
  &&  \quad 
   + D \check{\mathsf P}_1 \nabla (   
    \frac{1}{\check{\mathsf P}_1} {\mathsf P}_{\dt, 2} +    
  (-\Delta)^{-1} ( {\mathsf P}_{\dt, 2} - {\mathsf N}_{\dt, 2}  )  ) \Big)   
   - G_p^{(1)} ,   
     \label{consistency-7-2}      
\end{eqnarray}
and the solution depends only on the exact solution $({\mathsf N}, P)$, with derivatives of various orders stay bounded. In turn, an application of the semi-implicit discretization to~\eqref{consistency-7-1} -- \eqref{consistency-7-2} implies that
\begin{eqnarray} 
  \frac{{\mathsf N}_{\dt, 2}^{m+1} - {\mathsf N}_{\dt, 2}^m}{\dt}  &=& 
  \nabla \cdot \Big( 
    {\mathsf N}_{\dt, 2}^m \nabla ( 
    \ln \check{\mathsf N}_1^{m+1} +    
  (-\Delta)^{-1} ( \check{\mathsf N}_1^{m+1} - \check{\mathsf P}_1^{m+1} )  )    \nonumber 
\\
  &&   
   +   \check{\mathsf N}_1^m \nabla (   
    \frac{1}{\check{\mathsf N}_1^{m+1} } {\mathsf N}_{\dt, 2}^{m+1} 
    +    
  (-\Delta)^{-1} ( {\mathsf N}_{\dt, 2}^{m+1} - {\mathsf P}_{\dt, 2}^{m+1} )  ) \Big)   
   - ( G_n^{(1)} )^m  , 
     \label{consistency-8-1} 
\\
   \frac{{\mathsf P}_{\dt, 2}^{m+1} - {\mathsf P}_{\dt, 2}^m}{\dt}   &=& 
  \nabla \cdot \Big( 
  D {\mathsf P}_{\dt, 2}^m \nabla ( 
    \ln \check{\mathsf P}_1^{m+1} +    
  (-\Delta)^{-1} ( \check{\mathsf P}_1^{m+1} - \check{\mathsf N}_1^{m+1} )  )    \nonumber 
\\
  &&   
   + D \check{\mathsf P}_1^m \nabla (   
    \frac{1}{\check{\mathsf P}_1^{m+1} } {\mathsf P}_{\dt, 2}^{m+1} 
    +    
  (-\Delta)^{-1} ( {\mathsf P}_{\dt, 2}^{m+1} - {\mathsf N}_{\dt, 2}^{m+1}  )  ) \Big)   
   - ( G_p^{(1)} )^m .   
     \label{consistency-8-2}      
\end{eqnarray}
Subsequently, a combination of~\eqref{consistency-7-1} -- \eqref{consistency-7-2} and \eqref{consistency-8-1} -- \eqref{consistency-8-2} yields the third order temporal truncation error for $\check{\mathsf N}_2 := \check{\mathsf N}_1 + \dt^2 {\cal P}_N {\mathsf N}_{\dt, 2}$, $\check{\mathsf P}_2 := \check{\mathsf P}_1 + \dt^2 {\cal P}_N {\mathsf P}_{\dt, 2}$: 
\begin{eqnarray} 
  \frac{\check{\mathsf N}_2^{m+1} - \check{\mathsf N}_2^m}{\dt} &=& 
  \nabla \cdot \Big( 
    \check{\mathsf N}_2^m \nabla ( 
    \ln \check{\mathsf N}_2^{m+1} +    
  (-\Delta)^{-1} ( \check{\mathsf N}_2^{m+1} - \check{\mathsf P}_2^{m+1} )  ) \Big)   \nonumber 
\\
  &&
   + \dt^3 ( G_n^{(2)} )^m + O(\dt^3) + O (h^{m_0}) , 
     \label{consistency-9-1} 
\\ 
   \frac{\check{\mathsf P}_2^{m+1} - \check{\mathsf P}_2^m}{\dt} &=& 
    \nabla \cdot \Big( 
   D \check{\mathsf P}_2^m \nabla ( 
     \ln \check{\mathsf P}_2^{m+1} +    
  (-\Delta)^{-1} ( \check{\mathsf P}_2^{m+1} - \check{\mathsf N}_2^{m+1} )  )   \Big)   \nonumber 
\\
  &&
   + \dt^3 ( G_p^{(2)} )^m + O(\dt^3) + O (h^{m_0}) .  
     \label{consistency-9-2} 
\end{eqnarray} 
In fact, similar linearized expansions (as in~\eqref{consistency-6-1} -- \eqref{consistency-6-2}) have been used in the derivation. 

Next, we construct the spatial correction term $({\mathsf N}_{h, 1}, {\mathsf P}_{h, 1})$ to upgrade the spatial accuracy order. The following truncation error analysis for the spatial discretization can be obtained by using a straightforward Taylor expansion for the constructed profile  
$(\check{\mathsf N}_2 , \check{\mathsf P}_2)$: 
\begin{eqnarray} 
  \frac{\check{\mathsf N}_2^{m+1} - \check{\mathsf N}_2^m}{\dt} &=& 
  \nabla_h \cdot \Big( 
    {\cal A} ( \check{\mathsf N}_2^m ) \nabla_h ( 
    \ln \check{\mathsf N}_2^{m+1} +    
  (-\Delta_h)^{-1} ( \check{\mathsf N}_2^{m+1} - \check{\mathsf P}_2^{m+1} )  ) \Big)   \nonumber 
\\
  &&
   + \dt^3 ( G_n^{(2)} )^m + h^2 ( H_n^{(0)} )^m + O (\dt^3 + h^4) , 
     \label{consistency-10-1} 
\\ 
   \frac{\check{\mathsf P}_2^{m+1} - \check{\mathsf P}_2^m}{\dt} &=& 
    \nabla_h \cdot \Big( 
   D {\cal A} ( \check{\mathsf P}_2^m ) \nabla_h ( 
     \ln \check{\mathsf P}_2^{m+1} +    
  (-\Delta_h)^{-1} ( \check{\mathsf P}_2^{m+1} - \check{\mathsf N}_2^{m+1} )  )   \Big)   \nonumber 
\\
  &&
   + \dt^3 ( G_p^{(2)} )^m + h^2 ( H_p^{(0)} )^m + O (\dt^3 + h^4) ,  
     \label{consistency-10-2} 
\end{eqnarray} 
in which the average operator is taken in a similar form as~\eqref{mob ave-1}. The spatially discrete functions $H_n^{(0)}$, $H_p^{(0)}$ are smooth enough in the sense that their discrete derivatives are bounded. We also notice that there is no $O (h^3)$ truncation error term, due to the fact that the centered difference used in the spatial discretization gives local truncation errors with only even order terms, $O (h^2)$, $O (h^4)$, etc. Subsequently, the spatial correction function $({\mathsf N}_{h, 1}, {\mathsf P}_{h, 1})$ is given by solving the following linear PDE system: 
given by the following linear equations: 
\begin{eqnarray} 
  \partial_t {\mathsf N}_{h, 1}  &=& 
  \nabla \cdot \Big( 
    {\mathsf N}_{h, 1} \nabla ( 
    \ln \check{\mathsf N}_2 +    
  (-\Delta)^{-1} ( \check{\mathsf N}_2 - \check{\mathsf P}_2 )  )    \nonumber 
\\
  &&  \quad 
   +   \check{\mathsf N}_2 \nabla (   
    \frac{1}{\check{\mathsf N}_2} {\mathsf N}_{h, 1} +    
  (-\Delta)^{-1} ( {\mathsf N}_{h, 1} - {\mathsf P}_{h, 1} )  ) \Big)   
   - H_n^{(0)}  , 
     \label{consistency-11-1} 
\\
    \partial_t {\mathsf P}_{h, 1} &=& 
  \nabla \cdot \Big( 
  D {\mathsf P}_{h, 1} \nabla ( 
    \ln \check{\mathsf P}_2 +    
  (-\Delta)^{-1} ( \check{\mathsf P}_2 - \check{\mathsf N}_2 )  )    \nonumber 
\\
  &&  \quad 
   + D \check{\mathsf P}_1 \nabla (   
    \frac{1}{\check{\mathsf P}_2} {\mathsf P}_{h, 1} +    
  (-\Delta)^{-1} ( {\mathsf P}_{h, 1} - {\mathsf N}_{h, 1}  )  ) \Big)   
   - H_p^{(0)} ,   
     \label{consistency-11-2}      
\end{eqnarray}
and the solution depends only on the exact solution $({\mathsf N},  P)$, with the divided differences of various orders stay bounded. In turn, an application of a full discretization to \eqref{consistency-11-1} -- \eqref{consistency-11-2}  implies that
\begin{eqnarray} 
  \frac{{\mathsf N}_{h, 1}^{m+1} - {\mathsf N}_{h, 1}^m}{\dt}  &=& 
  \nabla_h \cdot \Big( 
    {\cal A} ( {\mathsf N}_{h, 1}^m ) \nabla_h ( 
    \ln \check{\mathsf N}_2^{m+1} +    
  (-\Delta_h)^{-1} ( \check{\mathsf N}_2^{m+1} - \check{\mathsf P}_2^{m+1} )  )    \nonumber 
\\
  &&  \, \,  
   +   {\cal A} ( \check{\mathsf N}_2^m ) \nabla (   
    \frac{1}{\check{\mathsf N}_2^{m+1} } {\mathsf N}_{h, 1}^{m+1} 
    +    
  (-\Delta_h)^{-1} ( {\mathsf N}_{h, 1}^{m+1} - {\mathsf P}_{h, 1}^{m+1} )  ) \Big)    \nonumber 
\\
  &&  \quad 
   - ( H_n^{(0)} )^m  + O (\dt + h^2) , 
     \label{consistency-12-1} 
\\
  \frac{{\mathsf P}_{h, 1}^{m+1} - {\mathsf P}_{h, 1}^m}{\dt}  &=& 
  \nabla_h \cdot \Big( 
  D {\cal A} ( {\mathsf P}_{h, 1}^m ) \nabla_h ( 
    \ln \check{\mathsf P}_2^{m+1} +    
  (-\Delta_h)^{-1} ( \check{\mathsf P}_2^{m+1} - \check{\mathsf N}_2^{m+1} )  )    \nonumber 
\\
  &&  \quad 
   + D {\cal A} ( \check{\mathsf P}_2^m ) \nabla_h (   
    \frac{1}{\check{\mathsf P}_2^{m+1} } {\mathsf P}_{h, 1}^{m+1} 
    +    
  (-\Delta_h)^{-1} ( {\mathsf P}_{h, 1}^{m+1} - {\mathsf N}_{h, 1}^{m+1}  )  ) \Big)    \nonumber 
\\
  &&  \quad 
   - ( H_p^{(0)} )^m + O (\dt + h^2 ) .   
     \label{consistency-12-2}      
\end{eqnarray}
Finally, a combination of~\eqref{consistency-11-1} -- \eqref{consistency-11-2} and~\eqref{consistency-12-1} -- \eqref{consistency-12-2} yields the higher order temporal truncation error for $(\check{\mathsf N}, \check{\mathsf P})$ (as given by~\eqref{consistency-1}): 
\begin{eqnarray} 
  \frac{\check{\mathsf N}^{m+1} - \check{\mathsf N}^m}{\dt} &=& 
  \nabla_h \cdot \Big( 
    {\cal A} ( \check{\mathsf N}^m ) \nabla_h ( 
    \ln \check{\mathsf N}^{m+1} +    
  (-\Delta_h)^{-1} ( \check{\mathsf N}^{m+1} - \check{\mathsf P}^{m+1} )  ) \Big) + \tau_n^{m+1} ,  \label{consistency-13-1} 
\\ 
   \frac{\check{\mathsf P}^{m+1} - \check{\mathsf P}^m}{\dt} &=& 
    \nabla_h \cdot \Big( 
   D{\cal A} ( \check{\mathsf P}^m ) \nabla_h ( 
     \ln \check{\mathsf P}^{m+1} +    
  (-\Delta_h)^{-1} ( \check{\mathsf P}^{m+1} - \check{\mathsf N}^{m+1} )  )   \Big) + \tau_p^{m+1} ,  \label{consistency-13-2} 
\end{eqnarray} 
where
	\[
\| \tau_n^{m+1} \|_2 , \| \tau_p^{m+1} \|_2 \le C (\dt^3 + h^4).
	\]
Again, the linear expansions have been extensively utilized.

	\begin{rem} 
Trivial initial data ${\mathsf N}_{\dt, j} (\, \cdot \, , t=0), {\mathsf P}_{\dt, j} (\, \cdot \, , t=0) \equiv 0$ are given ($j=1, 2$) as in \eqref{consistency-3-1} -- \eqref{consistency-3-2}, \eqref{consistency-7-1} -- \eqref{consistency-7-2}, respectively. Similar trivial initial data is also imposed to $({\mathsf N}_{h, 1},  {\mathsf P}_{h, 1})$ as in \eqref{consistency-11-1} -- \eqref{consistency-11-2}. Therefore, using similar arguments as in~\eqref{mass conserv-1-1} -- \eqref{mass conserv-2}, we conclude that
	\begin{equation} 
n^0 \equiv  \check{\mathsf N}^0 ,  \quad  p^0 \equiv  \check{\mathsf P}^0 ,  \quad \overline{n^k} = \overline{n^0} ,  \quad  \overline{p^k} = \overline{p^0} , \quad \forall \, k \ge 0, 
     \label{consistency-14-1}
     \end{equation}
and
	\begin{eqnarray}
\overline{\check{\mathsf N}^k} &=& \frac{1}{|\Omega|}\int_\Omega \, \check{\mathsf N} ( \cdot, t_k) \, d {\bf x}  = \frac{1}{|\Omega|}\int_\Omega \, \check{\mathsf N}^0 \, d {\bf x} = \overline{n^0} ,  \quad \forall \, k \ge 0,  
	\label{consistency-14-2}
	\\
\overline{\check{\mathsf P}^k} &=& \frac{1}{|\Omega|}\int_\Omega \, \check{\mathsf P} ( \cdot, t_k) \, d {\bf x}  = \frac{1}{|\Omega|}\int_\Omega \, \check{\mathsf P}^0 \, d {\bf x}  = \overline{p^0} ,  \quad \forall \, k \ge 0,  	
	\label{consistency-14-3}     
	\end{eqnarray}
where the first step of~\eqref{consistency-14-2} is based on the fact that $\check{\mathsf N} \in {\cal B}^K$, and the second step comes from the mass conservative property of $\check{\mathsf N}$ at the continuous level. These two properties will be used in later analysis. 

In addition, since $(\check{\mathsf N}, \check{\mathsf P})$ is mass conservative at a discrete level, as given by~\eqref{consistency-14-2} -- \eqref{consistency-14-3}, we observe that the local truncation error $\tau_n$, $\tau_p$ has a similar property: 
\begin{equation} 
  \overline{\tau_n^{m+1}} = \overline{\tau_p^{m+1}} = 0 ,  \quad \forall \, m \ge 0 . \label{consistency-14-4}
\end{equation}  
\end{rem}

\begin{rem} 
Since the temporal and spatial correction functions $({\mathsf N}_{\dt, j}, {\mathsf P}_{\dt, j})$, $({\mathsf N}_{h, 1}, {\mathsf P}_{h, 1})$ are bounded, we recall the separation property~\eqref{assumption:separation} for the exact solution, and obtain a similar property for the constructed profile $(\check{\mathsf N}, \check{\mathsf P})$: 
\begin{equation} 
  \check{\mathsf N} \ge \epsilon_0^\star , \, \, \, \check{\mathsf P}  \ge \epsilon_0^\star ,  
   \quad \mbox{for $\epsilon_0^\star > 0$} , 
    \label{assumption:separation-2}
	\end{equation}  
in which the projection estimate~\eqref{projection-est-0} has been repeatedly used. Notice that we could take $\dt$ and $h$ sufficiently small so that~\eqref{assumption:separation-2} is valid for a modified value $\epsilon_0^\star$, such as $\epsilon_0^\star  = \frac14 \epsilon_0$. Such a uniform bound will be used in the convergence analysis.
 
In addition, since the correction functions only depend on $({\mathsf N}_N, {\mathsf P}_N)$ and the exact solution, its $W^{1,\infty}$ norm will stay bounded. In turn, we are able to obtain a discrete $W^{1,\infty}$ bound for the  constructed profile $(\check{\mathsf N}, \check{\mathsf P})$:
\begin{equation} 
  \| \check{\mathsf N}^k \|_\infty \le C^\star , \, \, \, 
  \| \check{\mathsf P}^k \|_\infty \le C^\star  , \, \, \, 
  \| \nabla_h \check{\mathsf N}^k \|_\infty \le C^\star , \, \, \, 
  \| \nabla_h \check{\mathsf P}^k \|_\infty \le C^\star  ,  \quad \forall \, k \ge 0 .  
    \label{assumption:W1-infty bound}
	\end{equation}  
\end{rem}

\begin{rem}
The reason for such a higher order asymptotic expansion and truncation error estimate is to justify an a-priori $\ell^\infty$ bound of the numerical solution, which is needed to obtain the separation property, similarly formulated as~\eqref{assumption:separation-2} for the constructed approximate solution. With such a property valid for both the constructed approximate solution and the numerical solution, the nonlinear error term could be appropriately analyzed in the $\ell^\infty (0,T; \ell^2)$ convergence estimate. 
\end{rem}

\subsection{A rough error estimate} 

Instead of a direct analysis for the error function defined in~\eqref{error function-1}, we introduce alternate numerical error functions: 
	\begin{equation}
\tilde{n}^m := \mathcal{P}_h \check{\mathsf N}^m - n^m ,  \, \, \, 
\tilde{p}^m := \mathcal{P}_h \check{\mathsf P}^m - p^m , \, \, \, 
\tilde{\phi}^m :=   (-\Delta_h)^{-1} ( \tilde{p}^m - \tilde{n}^m ) , 
\quad \forall \ m \in \mathbb{N} .
	\label{error function-2}
	\end{equation}
The advantage of such a numerical error function is associated with its higher order accuracy, which comes from the higher order consistency estimate~\eqref{consistency-13-1} -- \eqref{consistency-13-2}. Again, since $\overline{\tilde{n}^m} = \overline{\tilde{p}^m} = 0$, which comes from the fact~\eqref{consistency-14-1} -- \eqref{consistency-14-3}, for any $m \ge 0$, we conclude  that the discrete norm $\nrm{ \, \cdot \, }_{-1,h}$ is well defined for the error grid function $(\tilde{n}^m, \tilde{p}^m)$. 

In turn, subtracting the numerical scheme~\eqref{scheme-PNP-1} -- \eqref{scheme-PNP-chem pot-p} from the consistency estimate~\eqref{consistency-13-1} -- \eqref{consistency-13-2} yields
\begin{eqnarray} 
  \frac{\tilde{n}^{m+1} - \tilde{n}^m}{\dt} &=& 	
  \nabla_h \cdot \Big( 
    {\cal A} ( n^m ) \nabla_h \tilde{\mu}_n^{m+1} 
  +   {\cal A} ( \tilde{n}^m ) \nabla_h {\cal V}_n^{m+1}  \Big) 
  + \tau_n^{m+1} , \label{error equation-1} 
\\
  \frac{\tilde{p}^{m+1} - \tilde{p}^m}{\dt} &=& 	
  \nabla_h \cdot \Big( 
  D {\cal A} ( p^m ) \nabla_h \tilde{\mu}_p^{m+1} 
  + D {\cal A} ( \tilde{p}^m ) \nabla_h {\cal V}_p^{m+1}  \Big) 
  + \tau_p^{m+1} , \label{error equation-2}   
	\end{eqnarray} 
where
	\begin{eqnarray} 
\tilde{\mu}_n^{m+1} &=& 	\ln \check{\mathsf N}^{m+1} - \ln n^{m+1}  +    (-\Delta_h)^{-1} ( \tilde{n}^{m+1} - \tilde{p}^{m+1} ) ,
	\label{error equation-3}   
	\\
{\cal V}_n^{m+1} &=&   \ln \check{\mathsf N}^{m+1} + (-\Delta_h)^{-1} ( \check{\mathsf N}^{m+1} - \check{\mathsf P}^{m+1} )   ,  
	\label{error equation-4}  
	\\
\tilde{\mu}_p^{m+1} &=& 	\ln \check{\mathsf P}^{m+1} - \ln p^{m+1} +    (-\Delta_h)^{-1} ( \tilde{p}^{m+1} - \tilde{n}^{m+1} ) ,
	\label{error equation-5}   
	\\
{\cal V}_p^{m+1} &=& \ln \check{\mathsf P}^{m+1} + (-\Delta_h)^{-1} ( \check{\mathsf P}^{m+1} - \check{\mathsf N}^{m+1} )   .  
	\label{error equation-6}  
	\end{eqnarray}

Since ${\cal V}_n^{m+1}$ and ${\cal V}_p^{m+1}$ only depend on the exact solution and the constructed profiles, we assume a discrete $W^{2,\infty}$ bound: 
	\begin{equation} 
\|  {\cal V}_n^{m+1} \|_{W_h^{2,\infty}} , \ \|  {\cal V}_p^{m+1} \|_{W_h^{2,\infty}}  \le C^\star . 
	\label{assumption:W2-infty bound}  
	\end{equation} 
To proceed with the nonlinear analysis, we make the following a-priori assumption at the previous time step: 
\begin{eqnarray} 
  \| \tilde{n}^m \|_2 , \, \| \tilde{p}^m \|_2 \le \dt^\frac{11}{4} + h^\frac{11}{4}  . 
   \label{a priori-1} 
\end{eqnarray} 
Such an a-priori assumption will be recovered by the optimal rate convergence analysis at the next time step, as will be demonstrated later. In turn, a discrete $W^{1,\infty}$ bound is available for the numerical error function at the previous time step, with the help of inverse inequality: 
	\begin{eqnarray} 
&& \| \tilde{n}^m \|_\infty  \le \frac{C \| \tilde{n}^m \|_2}{h^\frac32} \le   \frac{C ( \dt^\frac{11}{4} + h^\frac{11}{4} ) }{h^\frac32} \le  C ( \dt^\frac{5}{4} + h^\frac{5}{4} ) \le 1 ,
	\label{a priori-2} 
	\\
&&  \| \nabla_h \tilde{n}^m \|_\infty  \le \frac{C \| \tilde{n}^m \|_\infty}{h} \le   \frac{C ( \dt^\frac{5}{4} + h^\frac{5}{4} ) }{h} \le  C ( \dt^\frac{1}{4} + h^\frac{1}{4} ) \le 1  ,
	\label{a priori-3}    
	\end{eqnarray}
where the linear refinement constraint $C_1 h \le \dt \le C_h$ has been used. By similar arguments,
	\begin{equation}
\| \tilde{p}^m \|_\infty  \le  C ( \dt^\frac{5}{4} + h^\frac{5}{4} ) \le 1  \quad \mbox{and} \quad  \| \nabla_h \tilde{p}^m \|_\infty  \le  C ( \dt^\frac{1}{4} + h^\frac{1}{4} ) \le 1  . 
    \label{a priori-4}     
	\end{equation} 
Subsequently, the following $W_h^{1,\infty}$ bound is available for the numerical solution at the previous time step: 
	\begin{eqnarray} 
\| n^m \|_\infty &\le& \| \check{\mathsf N}^m \|_\infty +  \| \tilde{n}^m \|_\infty \le \tilde{C}_3 := C^\star +1 ,
	\label{a priori-5}
  	\\
\| p^m \|_\infty &\le& \| \check{\mathsf P}^m \|_\infty +  \| \tilde{p}^m \|_\infty  \le   \tilde{C}_3 ,
	\label{a priori-5-b} 
	\\
\| \nabla_h n^m \|_\infty &\le& \| \nabla_h \check{\mathsf N}^m \|_\infty  +  \| \nabla_h \tilde{n}^m \|_\infty  \le C^\star +1 = \tilde{C}_3 ,
	\label{a priori-6}  
	\\
\| \nabla_h p^m \|_\infty &\le& \| \nabla_h \check{\mathsf P}^m \|_\infty  +  \| \nabla_h \tilde{p}^m \|_\infty  \le  C^\star +1  = \tilde{C}_3 ,
	\label{a priori-7}   
	\end{eqnarray}           
with the regularity assumption~\eqref{assumption:W1-infty bound} applied. In addition, because of the $\ell^\infty$ estimate~\eqref{a priori-2}, \eqref{a priori-4} for the numerical error function, we can bound it by $\frac{\epsilon_0^\star}{2}$:  
	\begin{equation} 
\| \tilde{n}^m \|_\infty \le  C ( \dt^\frac{5}{4} + h^\frac{5}{4} ) \le \frac{\epsilon_0^\star}{2}  \quad \mbox{and} \quad \| \tilde{p}^m \|_\infty \le  C ( \dt^\frac{5}{4} + h^\frac{5}{4} ) \le \frac{\epsilon_0^\star}{2} ,
	\label{a priori-8}     
	\end{equation} 
so that the separation property is also valid for the numerical solution at the previous time step: 
	\begin{equation} 
n^m \ge \check{\mathsf N}^m -  \| \tilde{n}^m \|_\infty  \ge \frac{\epsilon_0^\star}{2} \quad \mbox{and} \quad  p^m \ge \check{\mathsf P}^m - \| \tilde{p}^m \|_\infty  \ge \frac{\epsilon_0^\star}{2} ,
    \label{assumption:separation-3}
	\end{equation}  
where the separation estimate~\eqref{assumption:separation-2} has been utilized. 

Taking a discrete inner product with~\eqref{error equation-1}, \eqref{error equation-2} by $\tilde{\mu}_n^{n+1}$, $\tilde{\mu}_p^{n+1}$, respectively, leads to 
	\begin{align} 
\langle \tilde{n}^{m+1} , \tilde{\mu}_n^{m+1} \rangle  & + \langle \tilde{p}^{m+1} , \tilde{\mu}_p^{m+1} \rangle + \dt (   \langle {\cal A} ( n^m )  \nabla_h \tilde{\mu}_n^{m+1} ,  \nabla_h \tilde{\mu}_n^{m+1}  \rangle + D \langle {\cal A} ( p^m )  \nabla_h \tilde{\mu}_p^{m+1} , \nabla_h \tilde{\mu}_p^{m+1}  \rangle  )
	\nonumber 
	\\
& = \langle \tilde{n}^m , \tilde{\mu}_n^{m+1} \rangle 
  + \langle \tilde{p}^m , \tilde{\mu}_p^{m+1} \rangle  
  + \dt  ( \langle \tau_n^{m+1} , \tilde{\mu}_n^{m+1} \rangle 
  + \langle \tau_p^{m+1} , \tilde{\mu}_p^{m+1} \rangle ) \nonumber 
\\
& \quad- \dt (   \langle {\cal A} ( \tilde{n}^m )  \nabla_h {\cal V}_n^{m+1} ,  
    \nabla_h \tilde{\mu}_n^{m+1}  \rangle   
  + D \langle {\cal A} ( \tilde{p}^m )  \nabla_h {\cal V}_p^{m+1} ,  
    \nabla_h \tilde{\mu}_p^{m+1}  \rangle  )   . 
  	\label{convergence-rough-1} 
	\end{align} 
Because of the separation estimate~\eqref{assumption:separation-3}, at a point-wise level, the following inequalities are available: 
	\begin{align}   
\langle {\cal A} ( n^m )  \nabla_h \tilde{\mu}_n^{m+1} ,  \nabla_h \tilde{\mu}_n^{m+1}  \rangle &\ge   \frac{\epsilon_0^\star}{2}  \| \nabla_h \tilde{\mu}_n^{m+1} \|_2^2 , 
    \label{convergence-rough-2}  
	\\
\langle {\cal A} ( p^m )  \nabla_h \tilde{\mu}_p^{m+1} ,  \nabla_h \tilde{\mu}_p^{m+1}  \rangle &\ge   \frac{\epsilon_0^\star}{2}  \| \nabla_h \tilde{\mu}_p^{m+1} \|_2^2 .  
    \label{convergence-rough-2-b}       
	\end{align}        
By the mean-free property~\eqref{consistency-14-4} for the local truncation error terms, the following estimate can be derived: 
	\begin{eqnarray} 
\langle \tau_n^{m+1} , \tilde{\mu}_n^{m+1} \rangle  &\le& \| \tau_n^{m+1} \|_{-1,h} \cdot \| \nabla_h \tilde{\mu}_n^{m+1} \|_2  \le   \frac{2}{  \epsilon_0^\star} \| \tau_n^{m+1} \|_{-1,h}^2  + \frac18   \epsilon_0^\star \| \nabla_h \tilde{\mu}_n^{m+1} \|_2^2 ,  
	\label{convergence-rough-3-1} 
	\\
\langle \tau_p^{m+1} , \tilde{\mu}_p^{m+1} \rangle  &\le& \| \tau_p^{m+1} \|_{-1,h} \cdot \| \nabla_h \tilde{\mu}_p^{m+1} \|_2 \le   \frac{2}{D \epsilon_0^\star}\| \tau_p^{m+1} \|_{-1,h}^2 + \frac18 D \epsilon_0^\star \| \nabla_h \tilde{\mu}_p^{m+1} \|_2^2 .   
	\label{convergence-rough-3-2}   
	\end{eqnarray}  
For the two terms $\langle \tilde{n}^m , \tilde{\mu}_n^{m+1} \rangle$ and 
$\langle \tilde{p}^m , \tilde{\mu}_p^{m+1} \rangle$, an application of the Cauchy inequality reveals that 
	\begin{eqnarray} 
\langle \tilde{n}^m , \tilde{\mu}_n^{m+1} \rangle &\le&  \| \tilde{n}^m \|_{-1,h} \cdot \| \nabla_h \tilde{\mu}_n^{m+1} \|_2 
  \le  \frac{2}{  \epsilon_0^\star \dt} \| \tilde{n}^m \|_{-1,h}^2 
  + \frac18   \epsilon_0^\star \dt \| \nabla_h \tilde{\mu}_n^{m+1} \|_2^2  ,  
	\label{convergence-rough-4-1} 
	\\
\langle \tilde{p}^m , \tilde{\mu}_p^{m+1} \rangle &\le&  \| \tilde{p}^m \|_{-1,h} \cdot \| \nabla_h \tilde{\mu}_p^{m+1} \|_2 
  \le  \frac{2}{D \epsilon_0^\star \dt} \| \tilde{p}^m \|_{-1,h}^2 
  + \frac18 D \epsilon_0^\star \dt \| \nabla_h \tilde{\mu}_p^{m+1} \|_2^2  .   
	\label{convergence-rough-4-2}   
	\end{eqnarray}    
For the last two terms on the right hand side of~\eqref{convergence-rough-1}, we see that 
	\begin{eqnarray}  
  -    \langle {\cal A} ( \tilde{n}^m )  \nabla_h {\cal V}_n^{m+1} ,  
    \nabla_h \tilde{\mu}_n^{m+1}  \rangle   
  &\le&   \| \nabla_h {\cal V}_n^{m+1}  \|_\infty 
  \cdot \| {\cal A} ( \tilde{n}^m )  \|_2 \cdot  \| \nabla_h \tilde{\mu}_n^{m+1}  \|_2  
  \nonumber 
\\
  &\le&
       C^\star    \|  \tilde{n}^m  \|_2 
  \cdot  \| \nabla_h \tilde{\mu}_n^{m+1}  \|_2    \nonumber 
\\
  &\le& 
  \frac{2 (C^\star)^2   }{\epsilon_0^\star}  \|  \tilde{n}^m  \|_2^2 
  + \frac18   \epsilon_0^\star \| \nabla_h \tilde{\mu}_n^{m+1} \|_2^2  , 
  \label{convergence-rough-5-1} 
	\end{eqnarray}
and, similarly,
	\begin{equation}
-  D \langle {\cal A} ( \tilde{p}^m )  \nabla_h {\cal V}_p^{m+1} , \nabla_h \tilde{\mu}_p^{m+1}  \rangle \le \frac{2 (C^\star)^2 D }{\epsilon_0^\star}  \|  \tilde{p}^m  \|_2^2  + \frac18 D \epsilon_0^\star \| \nabla_h \tilde{\mu}_p^{m+1} \|_2^2  .   
	\label{convergence-rough-5-2}   
	\end{equation} 
 A substitution of~\eqref{convergence-rough-2} -- \eqref{convergence-rough-5-2} into~\eqref{convergence-rough-1} leads to 
	\begin{align} 
\langle \tilde{n}^{m+1} , \tilde{\mu}_n^{m+1} \rangle & + \langle \tilde{p}^{m+1} , \tilde{\mu}_p^{m+1} \rangle  
  + \frac{\epsilon_0^\star}{2}  \dt (   \| \nabla_h \tilde{\mu}_n^{m+1} \|_2^2 
  + D  \| \nabla_h \tilde{\mu}_p^{m+1} \|_2^2 )
	\nonumber 
	\\
&\le  \frac{2}{  \epsilon_0^\star \dt} \| \tilde{n}^m \|_{-1,h}^2 
  + \frac{2}{D \epsilon_0^\star \dt} \| \tilde{p}^m \|_{-1,h}^2  
  + \frac{2 \dt}{  \epsilon_0^\star} \| \tau_n^{m+1} \|_{-1,h}^2 
  + \frac{2 \dt}{D \epsilon_0^\star}  \| \tau_p^{m+1} \|_{-1,h}^2
	\nonumber 
	\\
& \quad + 2 (C^\star)^2 (\epsilon_0^\star )^{-1} \dt (  \|  \tilde{n}^m  \|_2^2 + D^{-1}  \|  \tilde{p}^m  \|_2^2  ) .  
	\label{convergence-rough-6} 
	\end{align} 
Moreover, the detailed expansions in~\eqref{error equation-3} and \eqref{error equation-5} reveal the following identities: 
	\begin{eqnarray} 	
\langle \ln \check{\mathsf N}^{m+1} - \ln n^{m+1} , \tilde{n}^{m+1} \rangle &=& \langle \ln \check{\mathsf N}^{m+1} - \ln n^{m+1} , \check{\mathsf N}^{m+1} - n^{m+1}  \rangle   \ge 0 ,  
    \label{convergence-rough-7-1} 
	\\
\langle \ln \check{\mathsf P}^{m+1} - \ln p^{m+1} , \tilde{p}^{m+1} \rangle &=& \langle \ln \check{\mathsf P}^{m+1} - \ln p^{m+1} , 
   \check{\mathsf P}^{m+1} - p^{m+1}  \rangle   \ge 0 ,  
    \label{convergence-rough-7-2}         
	\end{eqnarray}
and
	\begin{eqnarray}
 \langle (-\Delta_h)^{-1} ( \tilde{n}^{m+1} - \tilde{p}^{m+1} ) , 
    \tilde{n}^{m+1} \rangle 
    + \langle (-\Delta_h)^{-1} ( \tilde{p}^{m+1} - \tilde{n}^{m+1} ) , 
    \tilde{p}^{m+1} \rangle & =& \|  \tilde{n}^{m+1} - \tilde{p}^{m+1} \|_{-1,h}^2 
	\nonumber
	\\
& \ge & 0 ,  
    \label{convergence-rough-7-3}        
	\end{eqnarray} 
where the positivities of $(n^{m+1}, p^{m+1})$ and $(\check{\mathsf N}^{m+1}, \check{\mathsf P}^{m+1})$ have been applied in the derivation of~\eqref{convergence-rough-7-1} and \eqref{convergence-rough-7-2}. Then we conclude that 
	\begin{equation} 
\langle \tilde{n}^{m+1} , \tilde{\mu}_n^{m+1} \rangle  + \langle \tilde{p}^{m+1} , \tilde{\mu}_p^{m+1} \rangle  \ge 0 . 
	\label{convergence-rough-7-4}        
	\end{equation}  
For the right hand side of~\eqref{convergence-rough-6}, the following estimates are available, which come from the a-priori assumption~\eqref{a priori-1}: 
	\begin{eqnarray} 
 \frac{2}{  \epsilon_0^\star \dt} \| \tilde{n}^m \|_{-1,h}^2  &\le& \frac{C}{  \epsilon_0^\star \dt} \| \tilde{n}^m \|_2^2  \le C ( \dt^\frac92 + h^\frac92 ) ,   
 	\label{convergence-rough-8-1}    
	\\
\frac{2}{D \epsilon_0^\star \dt} \| \tilde{p}^m \|_{-1,h}^2   &\le& \frac{C}{D \epsilon_0^\star \dt} \| \tilde{p}^m \|_2^2    
    \le C ( \dt^\frac92 + h^\frac92 )  , 
    \label{convergence-rough-8-2}  
	\\
\frac{2 \dt}{  \epsilon_0^\star} \| \tau_n^{m+1} \|_{-1,h}^2 &\le& C \dt \| \tau_n^{m+1} \|_2^2  \le C ( \dt^7 + \dt h^8) ,   
    \label{convergence-rough-8-3} 
	\\
\frac{2 \dt}{D \epsilon_0^\star}  \| \tau_p^{m+1} \|_{-1,h}^2 &\le& C \dt \| \tau_p^{m+1} \|_2^2 \le C ( \dt^7 + \dt h^8) ,  
    \label{convergence-rough-8-4}    
	\\
2 (C^\star)^2 (\epsilon_0^\star )^{-1} \dt  \|  \tilde{n}^m  \|_2^2 &\le& C \dt  \|  \tilde{n}^m  \|_2^2 \le C   ( \dt^\frac{13}{2} + h^\frac{13}{2} ) ,
	\label{convergence-rough-8-5}   
	\\
2 (C^\star)^2 (\epsilon_0^\star )^{-1}  D^{-1}  \dt \|  \tilde{p}^m  \|_2^2  &\le&  C \dt  \|  \tilde{p}^m  \|_2^2  \le C   ( \dt^\frac{13}{2} + h^\frac{13}{2} ) ,   
	\label{convergence-rough-8-6}      
	\end{eqnarray}
where the fact that $\| f \|_{-1,h} \le C \| f \|_2$, as well as the linear refinement constraint $C_1 h \le \dt \le C_2 h$, have been repeatedly applied. Going back~\eqref{convergence-rough-6}, we obtain 
	\begin{equation}
    \frac{\epsilon_0^\star}{2}  \dt (   \| \nabla_h \tilde{\mu}_n^{m+1} \|_2^2 
  + D  \| \nabla_h \tilde{\mu}_p^{m+1} \|_2^2 )  
  \le C ( \dt^\frac92 + h^\frac92 ) ,
  \label{convergence-rough-9-1} 
	\end{equation}
so that
	\begin{equation}
  \| \nabla_h \tilde{\mu}_n^{m+1} \|_2 ,  \| \nabla_h \tilde{\mu}_p^{m+1} \|_2  
  \le  C ( \dt^\frac74 + h^\frac74 )  .
	\label{convergence-rough-9-2}    
	\end{equation} 
Meanwhile, the error evolutionary equation~\eqref{error equation-1} implies that 
\begin{eqnarray} 
  \| \tilde{n}^{m+1} - \tilde{n}^m \|_2  	
  \le   \dt ( \| \nabla_h \cdot ( {\cal A} ( n^m ) \nabla_h \tilde{\mu}_n^{m+1} ) \|_2 
  + \| \nabla_h \cdot (  {\cal A} ( \tilde{n}^m ) \nabla_h {\cal V}_n^{m+1} ) \|_2  ) 
  + \dt \| \tau_n^{m+1} \|_2 .  \label{convergence-rough-10-1} 
\end{eqnarray}  
Furthermore, the following estimate is available for the first term, based on a detailed nonlinear expansion in the finite difference space, as well as repeated applications of discrete H\"older inequality: 
	\begin{eqnarray} 
\| \nabla_h \cdot ( {\cal A} ( n^m ) \nabla_h \tilde{\mu}_n^{m+1} ) \|_2 &\le& C ( \| n^m \|_\infty \cdot \| \nabla_h \nabla_h  \tilde{\mu}_n^{m+1} \|_2 +  \| \nabla_h n^m \|_\infty \cdot \| \nabla_h  \tilde{\mu}_n^{m+1} \|_2  )  
	\nonumber  
	\\
&\le& C \tilde{C}_3 ( \| \nabla_h \nabla_h  \tilde{\mu}_n^{m+1} \|_2  
  +  \| \nabla_h  \tilde{\mu}_n^{m+1} \|_2  )
	\nonumber
	\\
&\le& C  ( \dt^\frac34 + h^\frac34 ) ,   
    \label{convergence-rough-10-2} 
	\end{eqnarray}  
in which the a-priori estimates~\eqref{a priori-5}, \eqref{a priori-6} have been used in the second step, and the following inverse inequality has been applied in the last step: 
\begin{equation} 
   \| \nabla_h \nabla_h  \tilde{\mu}_n^{m+1} \|_2  
  \le  \frac{C  \| \nabla_h  \tilde{\mu}_n^{m+1} \|_2}{h}      
  \le C  ( \dt^\frac34 + h^\frac34 ) .  \label{convergence-rough-10-3} 
\end{equation}    
The second term on the right hand side of~\eqref{convergence-rough-10-1} could be similarly analyzed: 
	\begin{eqnarray} 
\| \nabla_h \cdot (  {\cal A} ( \tilde{n}^m ) \nabla_h {\cal V}_n^{m+1} ) \|_2 &\le& C ( \| \tilde{n}^m \|_2 \cdot \| \nabla_h \nabla_h  {\cal V}_n^{m+1} \|_\infty +  \| \nabla_h \tilde{n}^m \|_2 \cdot \| \nabla_h  {\cal V}_n^{m+1} \|_\infty  )  
	\nonumber  
	\\
	&\le& C C^\star (  \| \tilde{n}^m \|_2    
  +  \| \nabla_h  \tilde{n}^m \|_2  )
	\nonumber
	\\
&\le& C  ( \dt^\frac74 + h^\frac74 ) ,   
	\label{convergence-rough-10-4} 
	\end{eqnarray}  
in which the regularity assumption~\eqref{assumption:W2-infty bound} has been recalled in the second step, while an inverse inequality $\| \nabla_h f \|_2 \le \frac{C \| f \|_2}{h}$ has been applied in the last step. Therefore, a combination of~\eqref{convergence-rough-10-2}, \eqref{convergence-rough-10-4} and \eqref{convergence-rough-10-1} results in   
\begin{eqnarray} 
  \| \tilde{n}^{m+1} - \tilde{n}^m \|_2  	
  \le C  ( \dt^\frac74 + h^\frac74 ) + C  ( \dt^\frac{11}{4} + h^\frac{11}{4} ) 
  + C ( \dt^4 + \dt h^4) 
  \le C  ( \dt^\frac74 + h^\frac74 ) .  \label{convergence-rough-10-5} 
\end{eqnarray}  
A similar estimate could be derived for $\| \tilde{p}^{m+1} - \tilde{p}^m \|_2$: 
\begin{eqnarray} 
  \| \tilde{p}^{m+1} - \tilde{p}^m \|_2  	
   \le C  ( \dt^\frac74 + h^\frac74 ) .  \label{convergence-rough-10-6} 
\end{eqnarray}    
As a consequence, a combination with the a-priori error bound~\eqref{a priori-1} (at the previous time step) results in a rough error estimate for $\tilde{n}^{m+1}$, $\tilde{p}^{m+1}$: 
\begin{eqnarray} 
  \| \tilde{n}^{m+1} \|_2 + \| \tilde{p}^{m+1} \|_2  
  \le  \| \tilde{n}^m \|_2 + \| \tilde{p}^m \|_2    
  +  \| \tilde{n}^{m+1} - \tilde{n}^m \|_2    
  + \| \tilde{p}^{m+1} - \tilde{p}^m \|_2  	  
  \le \hat {C} ( \dt^\frac74 + h^\frac74 ) ,   \label{convergence-rough-11-1} 
\end{eqnarray}    
under the linear refinement requirement $C_1 h \le \dt \le C_2 h$, with $\hat{C}$ dependent on the physical parameters. Subsequently, an application of 3-D inverse inequality implies that 
	\begin{equation} 
\| \tilde{n}^{m+1} \|_\infty + \| \tilde{p}^{m+1} \|_\infty \le \frac{C ( \| \tilde{n}^{m+1} \|_2 + \| \tilde{p}^{m+1} \|_2 ) }{h^\frac32} \le \hat {C}_1 ( \dt^\frac14 + h^\frac14 )  ,
	\label{convergence-rough-11-2} 
	\end{equation}  
where $\hat{C}_1 := C \hat{C}$, under the same linear refinement requirement. Because of the accuracy order, we could take $\dt$ and $h$ sufficient small so that 
	\begin{equation} 
\hat {C}_1 ( \dt^\frac14 + h^\frac14 ) \le \frac{\epsilon_0^\star}{2} ,
	\end{equation}
so that
	\begin{equation}
\| \tilde{n}^{m+1} \|_\infty + \| \tilde{p}^{m+1} \|_\infty \le \frac{\epsilon_0^\star}{2}  .
	\label{convergence-rough-11-3}
	\end{equation}
Its combination with~\eqref{assumption:separation-2}, the separation property for the constructed approximate solution, leads to a similar property for the numerical solution at time step $t^{m+1}$:  
\begin{equation} 
  \frac{\epsilon_0^\star}{2} \le n^{m+1} 
  \le C^\star + \frac{\epsilon_0^\star}{2} \le \tilde{C}_3 \quad \mbox{and} \quad 
  \frac{\epsilon_0^\star}{2}  \le p^{m+1} 
  \le C^\star + \frac{\epsilon_0^\star}{2} \le \tilde{C}_3  .  
    \label{assumption:separation-4}
	\end{equation}  
Such a uniform $\| \cdot \|_\infty$ bound will play a very important role in the refined error estimate.

\begin{rem} 
In the rough error estimate~\eqref{convergence-rough-11-1}, we see that the accuracy order is lower than the one given by the a-priori-assumption~\eqref{a priori-1}. 
Therefore, such a rough estimate could not be used for a global induction analysis. Instead, the purpose of such an estimate is to establish a uniform $\| \cdot \|_\infty$ bound, via the technique of inverse inequality, so that a discrete separation property becomes available for the numerical solution, as well as its maximum values. With such a property established for the numerical solution, the refined error analysis will yield much sharper estimates. 
\end{rem}

\subsection{A refined error estimate} 

Before proceeding into the refined error estimate, the following two preliminary results are needed. 

\begin{lem} \label{prelim est-1}
Under the a-priori $\| \cdot \|_\infty$ estimate~\eqref{a priori-5}, \eqref{assumption:separation-3} for the numerical solution at the previous time step and the rough $\| \cdot \|_\infty$ estimate~\eqref{assumption:separation-4} for the one at the next time step, we have 
\begin{eqnarray}   
  \langle   {\cal A} ( n^m ) \nabla_h (  \ln \check{\mathsf N}^{m+1} - \ln n^{m+1} ) ,   
  \nabla_h \tilde{n}^{m+1} \rangle  
  &\ge& \gamma_n^{(0)}  \|  \nabla_h \tilde{n}^{m+1} \|_2^2  
  -  M_n^{(0)} \|  \tilde{n}^{m+1} \|_2^2   \nonumber 
\\
  && 
  - M_n^{(1)} h^8 ,   \label{nonlinear est-0-1} 
\\ 
   D  \langle  {\cal A} ( p^m ) \nabla_h (  \ln \check{\mathsf P}^{m+1} - \ln p^{m+1} ) ,   
  \nabla_h \tilde{p}^{m+1} \rangle  
  &\ge& \gamma_p^{(0)}  \|  \nabla_h \tilde{p}^{m+1} \|_2^2  
  -  M_p^{(0)} \|  \tilde{p}^{m+1} \|_2^2  \nonumber 
\\
  && 
  - M_p^{(1)} h^8 , \label{nonlinear est-0-2}   
	\end{eqnarray} 
where the constants $\gamma_n^{(0)}$, $\gamma_p^{(0)}$, $M_n^{(0)}$, $M_p^{(0)}$, $M_n^{(1)}$, $M_p^{(1)}$ only depend on $\epsilon_0^\star$, $C^\star$, $\tilde{C}_3$, $D$ and $| \Omega |$. 
\end{lem} 

\begin{proof} 
Looking at a single mesh cell $(i,j,k) \to (i+1,j,k)$, we make the following observation 
\begin{eqnarray}  
  D_x (  \ln \check{\mathsf N}^{m+1} - \ln n^{m+1} )_{i+\hf,j,k} 
  &=& \frac{1}{h} ( \ln {\mathsf N}_{i+1,j,k}^{m+1} - \ln {\mathsf N}_{i,j,k}^{m+1} ) 
  -  \frac{1}{h} ( \ln n_{i+1,j,k}^{m+1} - \ln n_{i,j,k}^{m+1} )  \nonumber 
\\
  &=& \frac{1}{\xi_{\mathsf N}} D_x {\mathsf N}_{i+\hf, j,k}^{m+1} 
  - \frac{1}{\xi_n} D_x n_{i+\hf, j,k}^{m+1}    \nonumber 
\\
  &=& \left( \frac{1}{\xi_{\mathsf N}} -  \frac{1}{\xi_n} \right) D_x {\mathsf N}_{i+\hf, j,k}^{m+1} 
  + \frac{1}{\xi_n} D_x \tilde{n}_{i+\hf, j,k}^{m+1}  ,  
  \label{nonlinear est-1-1}   
\end{eqnarray}   
in which the mean value theorem has been repeatedly applied, where
	\begin{equation} 
  \mbox{$\xi_{\mathsf N}$ is between ${\mathsf N}_{i+1,j,k}^{m+1}$ and ${\mathsf N}_{i,j,k}^{m+1}$ \quad   and  \quad  $\xi_n$ is between $n_{i+1,j,k}^{m+1}$ and $n_{i,j,k}^{m+1}$} .  
  \label{nonlinear est-1-2}  
	\end{equation} 
In turn, its product with $D_x \tilde{n}_{i+\hf,j,k}$ leads to 
	\begin{align} 
 D_x \tilde{n}_{i+\hf,j,k} \cdot D_x (  \ln \check{\mathsf N}^{m+1} &- \ln n^{m+1} )_{i+\hf,j,k}
	\nonumber 
	\\
& = \left( \frac{1}{\xi_{\mathsf N}} -  \frac{1}{\xi_n} \right) D_x {\mathsf N}_{i+\hf, j,k}^{m+1}  \cdot D_x \tilde{n}_{i+\hf,j,k} + \frac{1}{\xi_n} | D_x \tilde{n}_{i+\hf, j,k}^{m+1} |^2 .  
	\label{nonlinear est-2}   
	\end{align}     
For the second part, the rough $\| \cdot \|_\infty$ estimate~\eqref{assumption:separation-4} for $n^{m+1}$ implies that $0 < \xi_n \le \tilde{C}_3$, which in turn gives 
	\begin{equation}  
\frac{1}{\xi_n} \ge \frac{1}{\tilde{C}_3} \quad \mbox{and} \quad \frac{1}{\xi_n} | D_x \tilde{n}_{i+\hf, j,k}^{m+1} |^2 \ge  \frac{1}{\tilde{C}_3}  | D_x \tilde{n}_{i+\hf, j,k}^{m+1} |^2 .     
	\label{nonlinear est-3}   
	\end{equation} 
For the first term on the right hand side of~\eqref{nonlinear est-2}, we begin with the following identity: 
\begin{eqnarray} 
  \frac{1}{\xi_{\mathsf N}} 
  = \frac{ \ln {\mathsf N}_{i+1,j,k}^{m+1} - \ln {\mathsf N}_{i,j,k}^{m+1} }{ 
  {\mathsf N}_{i+1,j,k}^{m+1} -  {\mathsf N}_{i,j,k}^{m+1} }  
  = \frac{ \ln \Big( 1 + \frac{ {\mathsf N}_{i+1,j,k}^{m+1} - {\mathsf N}_{i,j,k}^{m+1}  }{ 
  {\mathsf N}_{i,j,k}^{m+1} } \Big) }{ {\mathsf N}_{i+1,j,k}^{m+1} -  {\mathsf N}_{i,j,k}^{m+1} }   . 
  \label{nonlinear est-4-1}     
\end{eqnarray} 
By setting $t_N^{(0)} = \frac{ {\mathsf N}_{i+1,j,k}^{m+1} - {\mathsf N}_{i,j,k}^{m+1}  }{ 
  {\mathsf N}_{i,j,k}^{m+1} }$, the following Taylor expansion is available:  
\begin{eqnarray} 
  \ln ( 1 + t_N^{(0)} ) = t_N^{(0)} - \frac12 ( t_N^{(0)} )^2 + \frac13 ( t_N^{(0)} )^3   
  - \frac14 ( t_N^{(0)} )^4 + \frac{1}{5 ( 1 + \eta_N)^5} ( t_N^{(0)} )^5 ,  
     \label{nonlinear est-4-2}     
\end{eqnarray} 
with $\eta_N$ between 0 and $t_N^{(0)}$. Its substitution into~\eqref{nonlinear est-4-1} yields
\begin{eqnarray} 
  \frac{1}{\xi_{\mathsf N}} 
  &=& \frac{ 1 }{ {\mathsf N}_{i,j,k}^{m+1} }
 - \frac{ {\mathsf N}_{i+1,j,k}^{m+1} - {\mathsf N}_{i,j,k}^{m+1}  }{ 
  2 ( {\mathsf N}_{i,j,k}^{m+1} )^2 } 
  + \frac{ ( {\mathsf N}_{i+1,j,k}^{m+1} - {\mathsf N}_{i,j,k}^{m+1} )^2 }{ 
  3 ( {\mathsf N}_{i,j,k}^{m+1} )^3 }
  - \frac{ ( {\mathsf N}_{i+1,j,k}^{m+1} - {\mathsf N}_{i,j,k}^{m+1} )^3 }{ 
  4 ( {\mathsf N}_{i,j,k}^{m+1} )^4 }   \nonumber 
\\
  && 
  + \frac{1}{5 ( 1 + \eta_N)^5} 
  \frac{ ( {\mathsf N}_{i+1,j,k}^{m+1} - {\mathsf N}_{i,j,k}^{m+1} )^4 }{ 
    ( {\mathsf N}_{i,j,k}^{m+1} )^5 }   .     
  \label{nonlinear est-4-3}     
\end{eqnarray} 
A similar equality could be derived for $\frac{1}{\xi_n}$: 
\begin{eqnarray} 
  \frac{1}{\xi_n} 
  &=& \frac{ 1 }{ n_{i,j,k}^{m+1} }
 - \frac{ n_{i+1,j,k}^{m+1} - n_{i,j,k}^{m+1}  }{ 
  2 ( n_{i,j,k}^{m+1} )^2 } 
  + \frac{ ( n_{i+1,j,k}^{m+1} - n_{i,j,k}^{m+1} )^2 }{ 
  3 ( n_{i,j,k}^{m+1} )^3 }
  - \frac{ ( n_{i+1,j,k}^{m+1} - n_{i,j,k}^{m+1} )^3 }{ 
  4 ( n_{i,j,k}^{m+1} )^4 }   \nonumber 
\\
  && 
  + \frac{1}{5 ( 1 + \eta_n)^5} \frac{ ( n_{i+1,j,k}^{m+1} - n_{i,j,k}^{m+1} )^4 }{ 
    ( n_{i,j,k}^{m+1} )^5 }  ,      
  \label{nonlinear est-4-3}     
\end{eqnarray} 
with $\eta_n$ between 0 and $t_n^{(0)} = \frac{ n_{i+1,j,k}^{m+1} - n_{i,j,k}^{m+1}  }{ 
  n_{i,j,k}^{m+1} }$. In addition, the following estimates are derived: 
	\begin{equation} 
\Big| \frac{ 1 }{ {\mathsf N}_{i,j,k}^{m+1} } -  \frac{ 1 }{ n_{i,j,k}^{m+1} }  \Big| 
     =  \Big| \frac{ \tilde{n}_{i,j,k}^{m+1} }{ {\mathsf N}_{i,j,k}^{m+1} n_{i,j,k}^{m+1} }  \Big| 
     \le \frac{2}{( \epsilon_0^\star)^2}  | \tilde{n}_{i,j,k}^{m+1} | ,  
     \label{nonlinear est-5-1} 
     \end{equation}
and
	\begin{align}
\Big| \frac{ {\mathsf N}_{i+1,j,k}^{m+1} - {\mathsf N}_{i,j,k}^{m+1}  }{ 
  ( {\mathsf N}_{i,j,k}^{m+1} )^2 }   
  -  \frac{ n_{i+1,j,k}^{m+1} - n_{i,j,k}^{m+1}  }{ 
   ( n_{i,j,k}^{m+1} )^2 }  \Big| &\le
      \Big| \frac{ \tilde{n}_{i+1,j,k}^{m+1} - \tilde{n}_{i,j,k}^{m+1}  }{ 
  ( n_{i,j,k}^{m+1} )^2 }  \Big|
  	\nonumber
  	\\
& \quad +  \Big| \frac{ ( {\mathsf N}_{i+1,j,k}^{m+1} - {\mathsf N}_{i,j,k}^{m+1} ) 
   ( n_{i+1,j,k}^{m+1} - {\mathsf N}_{i,j,k}^{m+1} )  \tilde{n}_{i,j,k}^{m+1} }{ 
  ( {\mathsf N}_{i,j,k}^{m+1} )^2 ( n_{i,j,k}^{m+1} )^2}  \Big|   
	\nonumber 
	\\
&\le \frac{4}{( \epsilon_0^\star)^2} ( | \tilde{n}_{i,j,k}^{m+1} |  + | \tilde{n}_{i+1,j,k}^{m+1} |  ) 
	\nonumber
	\\
& \quad +  \frac{ 2 C^\star (C^\star + \tilde{C}_3 )  }{\frac14 ( \epsilon_0^\star )^4 } | \tilde{n}_{i,j,k}^{m+1} |     
	\nonumber
	\\
& \le Q^{(2)} ( | \tilde{n}_{i,j,k}^{m+1} |  + | \tilde{n}_{i+1,j,k}^{m+1} |  )  , 
  \label{nonlinear est-5-2}   
\end{align}         
where
	 \[
Q^{(2)} := \frac{4}{( \epsilon_0^\star)^2} +  \frac{ 8 C^\star (C^\star + \tilde{C}_3 )  }{( \epsilon_0^\star )^4 },
	\]
and the rough $\| \cdot \|_\infty$ estimate~\eqref{assumption:separation-4}, the regularity assumption~\eqref{assumption:W1-infty bound}, and the separation property~\eqref{assumption:separation-2} have been extensively applied. The two other difference terms could be similarly analyzed: 
\begin{eqnarray} 
  &&
   \Big| \frac{ ( {\mathsf N}_{i+1,j,k}^{m+1} - {\mathsf N}_{i,j,k}^{m+1} )^2 }{ 
   ( {\mathsf N}_{i,j,k}^{m+1} )^3 } 
   -  \frac{ ( n_{i+1,j,k}^{m+1} - n_{i,j,k}^{m+1} )^2 }{ 
   ( n_{i,j,k}^{m+1} )^3 } \Big| 
  \le Q^{(3)} ( | \tilde{n}_{i,j,k}^{m+1} |  + | \tilde{n}_{i+1,j,k}^{m+1} |  )  , 
  \label{nonlinear est-5-3}        
\\
  &&   
   \Big| \frac{ ( {\mathsf N}_{i+1,j,k}^{m+1} - {\mathsf N}_{i,j,k}^{m+1} )^3 }{ 
  ( {\mathsf N}_{i,j,k}^{m+1} )^4 }   \nonumber 
  - \frac{ ( n_{i+1,j,k}^{m+1} - n_{i,j,k}^{m+1} )^3 }{ 
  ( n_{i,j,k}^{m+1} )^4 }   \Big| 
  \le Q^{(4)} ( | \tilde{n}_{i,j,k}^{m+1} |  + | \tilde{n}_{i+1,j,k}^{m+1} |  )  , 
  \label{nonlinear est-5-4}      
	\end{eqnarray} 
where $Q^{(3)}$, $Q^{(4)}$ only depend on $\epsilon_0^\star$, $C^\star$ and $\tilde{C}_3$. For the remainder terms, we observe that 
	\begin{align} 
| {\mathsf N}_{i+1,j,k}^{m+1}  - {\mathsf N}_{i,j,k}^{m+1} | & = h | D_x  {\mathsf N}_{i+\hf,j,k}^{m+1} | \le h \| D_x  {\mathsf N}^{m+1} \|_\infty \le C^\star h ,
	\label{nonlinear est-5-5}  
	\\
| t_N^{(0)} | & = \Big| \frac{ {\mathsf N}_{i+1,j,k}^{m+1} - {\mathsf N}_{i,j,k}^{m+1}  }{ 
  {\mathsf N}_{i,j,k}^{m+1} }   \Big|  
  \le C^\star ( \epsilon_0^\star )^{-1} h  \le Q^{(5)} h  \le \frac12,
	\label{nonlinear est-5-5} 
	\end{align}
where $Q^{(5)} =  C^\star ( \epsilon_0^\star )^{-1}$ and where we have used $\epsilon_0^\star \le {\mathsf N}_{i,j,k}^{m+1}$. Furthermore $| \eta_N | \le \frac12$,  so that
	\begin{equation}
| 1 + \eta_N | \ge \frac12 \quad  \mbox{and}  \quad \Big| \frac{1}{5 ( 1 + \eta_N)^5}  \Big| \le \frac{32}{5} . 
  \label{nonlinear est-5-5}     
	\end{equation}
Finally,
	\begin{equation}    
| {\cal R}_1 | = \Big| \frac{1}{5 ( 1 + \eta_N)^5} 
  \frac{ ( {\mathsf N}_{i+1,j,k}^{m+1} - {\mathsf N}_{i,j,k}^{m+1} )^4 }{ 
    ( {\mathsf N}_{i,j,k}^{m+1} )^5 }   \Big|   
    \le \frac{32}{5} \cdot \frac{ (C^\star h)^4 }{ ( \epsilon_0^\star )^5}  
    \le Q^{(6)} h^4 ,   
    \label{nonlinear est-5-6}        
	\end{equation} 
with $Q^{(6)} =  \frac{32 (C^\star)^4}{5 ( \epsilon_0^\star )^5}$. The other remainder term has a similar bound
	\begin{equation} 
     | {\cal R}_2 | = \Big| \frac{1}{5 ( 1 + \eta_n)^5} 
  \frac{ ( n_{i+1,j,k}^{m+1} - n_{i,j,k}^{m+1} )^4 }{ 
    ( n_{i,j,k}^{m+1} )^5 }   \Big|   
    \le \frac{32}{5} \cdot \frac{ (\tilde{C}_3 h)^4 }{ \frac{1}{32} ( \epsilon_0^\star )^5}  
    \le Q^{(7)} h^4 ,   \label{nonlinear est-5-7}        
	\end{equation} 
with $Q^{(7)} =  \frac{1024 \tilde{C}_3^4}{5 ( \epsilon_0^\star )^5}$. Consequently, a combination of~\eqref{nonlinear est-5-1} -- \eqref{nonlinear est-5-4}, \eqref{nonlinear est-5-6} and \eqref{nonlinear est-5-7} indicates that 
\begin{eqnarray} 
  \Big|   \frac{1}{\xi_{\mathsf N}} -  \frac{1}{\xi_n} \Big|   
  \le Q^{(0)}   ( | \tilde{n}_{i,j,k}^{m+1} |  + | \tilde{n}_{i+1,j,k}^{m+1} |  ) 
 +  Q^\star h^4 ,   \label{nonlinear est-5-8}        
\end{eqnarray} 
with 
	\[
Q^{(0)} = \frac{2}{( \epsilon_0^\star)^2} + \frac12 Q^{(2)} + \frac13 Q^{(3)} + \frac14 Q^{(4)} \quad \mbox{and} \quad Q^\star = Q^{(6)} + Q^{(7)}.
	\]
Then we arrive at an estimate for the first part on the right hand side of~\eqref{nonlinear est-2}: 
	\begin{eqnarray} 
&& \hspace{-0.5in} \left( \frac{1}{\xi_{\mathsf N}} -  \frac{1}{\xi_n} \right) 
   D_x {\mathsf N}_{i+\hf, j,k}^{m+1}  \cdot D_x \tilde{n}_{i+\hf,j,k}    \nonumber 
\\
  &\ge& 
      - ( Q^{(0)}   ( | \tilde{n}_{i,j,k}^{m+1} |  + | \tilde{n}_{i+1,j,k}^{m+1} |  ) 
 +  Q^\star h^4 ) \cdot C^\star \cdot |  D_x \tilde{n}_{i+\hf,j,k} |     \nonumber 
\\
  &\ge& 
      - ( Q^{(0)}   ( | \tilde{n}_{i,j,k}^{m+1} |  + | \tilde{n}_{i+1,j,k}^{m+1} |  ) 
 +  Q^\star h^4 )^2 (C^\star)^2 \tilde{C}_3  
    - ( 4 \tilde{C}_3 )^{-1}  | D_x \tilde{n}_{i+\hf, j,k}^{m+1} |^2  . 
    \label{nonlinear est-6-1}   
\end{eqnarray}  
Subsequently, a combination of~\eqref{nonlinear est-2}, \eqref{nonlinear est-2} and \eqref{nonlinear est-6-1} results in 
	\begin{eqnarray} 
&& \hspace{-0.5in}
    D_x \tilde{n}_{i+\hf,j,k} 
    \cdot D_x (  \ln \check{\mathsf N}^{m+1} - \ln n^{m+1} )_{i+\hf,j,k}  \nonumber 
\\
  &\ge& 
  - ( Q^{(0)}   ( | \tilde{n}_{i,j,k}^{m+1} |  + | \tilde{n}_{i+1,j,k}^{m+1} |  ) 
 +  Q^\star h^4 )^2 (C^\star)^2 \tilde{C}_3   
  + \frac{3}{4 \tilde{C}_3}  | D_x \tilde{n}_{i+\hf, j,k}^{m+1} |^2  \nonumber 
\\
   &\ge& 
   \frac{3}{4 \tilde{C}_3}  | D_x \tilde{n}_{i+\hf, j,k}^{m+1} |^2      
  - 4 ( Q^{(0)} C^\star )^2 \tilde{C}_3  
  ( | \tilde{n}_{i,j,k}^{m+1} |^2  + | \tilde{n}_{i+1,j,k}^{m+1} |^2  ) 
 - 2  ( Q^\star C^\star)^2 \tilde{C}_3  h^8 .  
 \label{nonlinear est-6-2}   
\end{eqnarray}     
Notice that this inequality is valid at a point-wise level. With summation over space, and keeping in mind of the a-priori $\| \cdot \|_\infty$ estimate~\eqref{a priori-5}, \eqref{assumption:separation-3} for $n^m$, we obtain
	\begin{eqnarray}  
&& \hspace{-0.5in} \langle   {\cal A} ( n^m ) \nabla_h (  \ln \check{\mathsf N}^{m+1} - \ln n^{m+1} ) ,    \nabla_h \tilde{n}^{m+1} \rangle
	\nonumber 
	\\
  &\ge& 
     \frac{\epsilon_0^\star}{2} \cdot   
  \frac{3}{4 \tilde{C}_3}  \| \nabla_h \tilde{n}^{m+1} |_2^2  
  - 8 ( Q^{(0)} C^\star )^2 \tilde{C}_3^2    
   \| \tilde{n}^{m+1} \|_2^2  
 - 2  ( Q^\star C^\star)^2 \tilde{C}_3^2 | \Omega |  h^8  .  
 \end{eqnarray} 
 This proves the first nonlinear estimate~\eqref{nonlinear est-0-1}, by setting $\gamma_n^{(0)} = \frac{3 \epsilon_0^\star}{8 \tilde{C}_3}$, $M_n^{(0)} = 8 ( Q^{(0)} C^\star )^2 \tilde{C}_3^2$, and $M_n^{(1)} =  2  ( Q^\star C^\star)^2 \tilde{C}_3^2 | \Omega |$. The second nonlinear estimate~\eqref{nonlinear est-0-2} could be derived exactly in the same manner. The details are skipped for the sake of brevity. 
\end{proof} 

The next preliminary estimate is more straightforward. 

	\begin{lem}
	\label{prelim est-2}
For $\tilde{\phi}^k$ (for any $k \ge 0$) defined in~\eqref{error function-2},  we have the estimate
\begin{equation} 
  \| \nabla_h \tilde{\phi}^k \|_2 \le \tilde{C}_4 
  \| \tilde{n}^k - \tilde{p}^k \|_2   ,  
  \label{prelim est-2-0}  
\end{equation} 
for some constant $\tilde{C}_4>0$ that is independent of $h$.
	\end{lem} 

	\begin{proof} 
  Inequality~\eqref{prelim est-2-0} is a direct consequence of the standard estimate: $\| f \|_{-1,h} \le C \| f \|_2$, for any $f$ with $\overline{f} =0$. 
	\end{proof}

Now we proceed with the refined error estimate. Taking a discrete inner product with~\eqref{error equation-1}, \eqref{error equation-2} by $2 \tilde{n}^{m+1}$, $2 \tilde{p}^{m+1}$, respectively, leads to 
\begin{eqnarray} 
  &&
  \frac{1}{\dt} ( \| \tilde{n}^{m+1} \|_2^2 - \| \tilde{n}^m \|_2^2 
  +   \| \tilde{n}^{m+1}  - \tilde{n}^m \|_2^2   
  + \| \tilde{p}^{m+1} \|_2^2 - \| \tilde{p}^m \|_2^2 
  +   \| \tilde{p}^{m+1}  - \tilde{p}^m \|_2^2  )    \nonumber 
\\
  && 
  +  2 (   \langle {\cal A} ( n^m )  \nabla_h \tilde{\mu}_n^{m+1} ,  
      \nabla_h \tilde{n}^{m+1}  \rangle   
  + D \langle {\cal A} ( p^m )  \nabla_h \tilde{\mu}_p^{m+1} ,  
     \nabla_h \tilde{p}^{m+1}  \rangle  )    \nonumber 
\\ 
  &=&
  2  ( \langle \tau_n^{m+1} , \tilde{n}^{m+1} \rangle 
  + \langle \tau_p^{m+1} , \tilde{p}^{m+1} \rangle ) \nonumber 
\\
  && 
  - 2 (   \langle {\cal A} ( \tilde{n}^m )  \nabla_h {\cal V}_n^{m+1} ,  
     \nabla_h \tilde{n}^{m+1}  \rangle   
  + D \langle {\cal A} ( \tilde{p}^m )  \nabla_h {\cal V}_p^{m+1} ,  
    \nabla_h \tilde{p}^{m+1}  \rangle  )  ,  
  \label{convergence-1} 
\end{eqnarray}  
where summation-by-parts has been applied. For the local truncation error terms, similar estimates could be derived:  
\begin{eqnarray} 
  2 \langle \tau_n^{m+1} , \tilde{n}^{m+1} \rangle  
  \le  \| \tau_n^{m+1} \|_2^2 + \| \tilde{n}^{m+1} \|_2^2 , \, \, \, 
  2 \langle \tau_p^{m+1} , \tilde{p}^{m+1} \rangle  
  \le  \| \tau_p^{m+1} \|_2^2 + \| \tilde{p}^{m+1} \|_2^2 . \label{convergence-2} 
\end{eqnarray}  
For the nonlinear diffusion error inner product on the left hand side, we see that 
\begin{eqnarray} 
     \langle {\cal A} ( n^m )  \nabla_h \tilde{\mu}_n^{m+1} ,  
      \nabla_h \tilde{n}^{m+1}  \rangle   
      &=&     \langle   {\cal A} ( n^m ) \nabla_h (  \ln \check{\mathsf N}^{m+1} - \ln n^{m+1} ) ,   
  \nabla_h \tilde{n}^{m+1} \rangle   \nonumber 
\\
  && 
  +  \langle   {\cal A} ( n^m ) \nabla_h \tilde{\phi}^{m+1} ,   
  \nabla_h \tilde{n}^{m+1} \rangle .  \label{convergence-3-1} 
\end{eqnarray}    
The second part has the following lower bound 
	\begin{eqnarray} 
 \langle   {\cal A} ( n^m ) \nabla_h \tilde{\phi}^{m+1} ,   
  \nabla_h \tilde{n}^{m+1} \rangle 
  &\ge& - \tilde{C}_3 \| \nabla_h \tilde{\phi}^{m+1} \|_2 
  \cdot \|  \nabla_h \tilde{n}^{m+1} \|_2  \nonumber 
\\
  &\ge& - \tilde{C}_3 \tilde{C}_4 
   \| \tilde{n}^{m+1} - \tilde{p}^{m+1} \|_2 
  \cdot \|  \nabla_h \tilde{n}^{m+1} \|_2  \nonumber 
\\
  &\ge& -  (\tilde{C}_3 \tilde{C}_4 )^2
   \| \tilde{n}^{m+1} - \tilde{p}^{m+1} \|_2^2  - \frac14 \gamma^{(0)} \|  \nabla_h \tilde{n}^{m+1} \|_2,  \nonumber 
  \label{convergence-3-2}   
\end{eqnarray} 
in which the inequality~\eqref{prelim est-2-0} (in Lemma~\ref{prelim est-2}) has been applied in the second step. Its substitution into~\eqref{convergence-3-1},  combined with the preliminary estimate~\eqref{nonlinear est-0-1} (in Lemma~\ref{prelim est-1}), leads to 
	\begin{eqnarray} 
  &&  \hspace{-0.5in}
     \langle {\cal A} ( n^m )  \nabla_h \tilde{\mu}_n^{m+1} ,  
      \nabla_h \tilde{n}^{m+1}  \rangle    \nonumber 
 \\
     \hspace{-0.4in} &\ge&   
      \frac34 \gamma_n^{(0)}  \|  \nabla_h \tilde{n}^{m+1} \|_2^2  
  -  M_n^{(0)} \|  \tilde{n}^{m+1} \|_2^2   
  - M_n^{(1)} h^8       
     -  2 ( \tilde{C}_3 \tilde{C}_4 )^2
   (  \| \tilde{n}^{m+1} \|_2^2 + \| \tilde{p}^{m+1} \|_2^2 ) .  \label{convergence-3-3} 
\end{eqnarray}    
A similar lower bound could be derived for the other nonlinear error inner product on the left hand side; the details are skipped for the sake of brevity: 
\begin{eqnarray} 
   && \hspace{-0.5in}
   D \langle {\cal A} ( p^m )  \nabla_h \tilde{\mu}_p^{m+1} ,  
      \nabla_h \tilde{p}^{m+1}  \rangle    \nonumber 
 \\
     \hspace{-0.4in} &\ge&   
      \frac34 \gamma_p^{(0)}  \|  \nabla_h \tilde{p}^{m+1} \|_2^2  
  -  M_p^{(0)} \|  \tilde{p}^{m+1} \|_2^2   
  - M_p^{(1)} h^8       
     -  2 ( D \tilde{C}_3 \tilde{C}_4 )^2
   (  \| \tilde{n}^{m+1} \|_2^2 + \| \tilde{p}^{m+1} \|_2^2 ) .  \label{convergence-3-4} 
\end{eqnarray}    
For the last two nonlinear error inner product terms on the right hand side, a direct application of Cauchy inequality is applied: 
\begin{eqnarray} 
    - 2    \langle {\cal A} ( \tilde{n}^m )  \nabla_h {\cal V}_n^{m+1} ,  
     \nabla_h \tilde{n}^{m+1}  \rangle  
     &\le&   2   \| \nabla_h {\cal V}_n^{m+1} \|_\infty 
     \cdot  \| {\cal A} ( \tilde{n}^m )  \|_2  
     \cdot \| \nabla_h \tilde{n}^{m+1}  \|_2  \nonumber  
\\
  &\le&  
    2   C^\star \| \tilde{n}^m   \|_2  
     \cdot \| \nabla_h \tilde{n}^{m+1}  \|_2   \nonumber 
\\
   &\le&   2 (   C^\star )^2 (\gamma_n^{(0)} )^{-1} \| \tilde{n}^m   \|_2^2  
     + \frac12 \gamma_n^{(0)} \| \nabla_h \tilde{n}^{m+1}  \|_2^2  , 
      \label{convergence-4-1}  
\\
   - 2  D \langle {\cal A} ( \tilde{p}^m )  \nabla_h {\cal V}_p^{m+1} ,  
     \nabla_h \tilde{p}^{m+1}  \rangle  
    &\le&   2 ( D C^\star )^2 (\gamma_p^{(0)} )^{-1} \| \tilde{p}^m   \|_2^2  
     + \frac12 \gamma_p^{(0)} \| \nabla_h \tilde{p}^{m+1}  \|_2^2  ,  
      \label{convergence-4-2}               
\end{eqnarray}      
with the regularity assumption~\eqref{assumption:W2-infty bound} recalled.  
     
Finally, a substitution of~\eqref{convergence-2}, \eqref{convergence-3-3} -- \eqref{convergence-3-4} and \eqref{convergence-4-1} -- \eqref{convergence-4-2} into~\eqref{convergence-1} results in 
	\begin{eqnarray} 
&& \hspace{-0.5in} \frac{1}{\dt} ( \| \tilde{n}^{m+1} \|_2^2 - \| \tilde{n}^m \|_2^2   
  + \| \tilde{p}^{m+1} \|_2^2 - \| \tilde{p}^m \|_2^2  )    
  + \gamma_n^{(0)} \| \nabla_h \tilde{n}^{m+1}  \|_2^2  
  + \gamma_p^{(0)} \| \nabla_h \tilde{p}^{m+1}  \|_2^2   
	\nonumber 
	\\ 
&\le& 
  M^{(2)}  (  \| \tilde{n}^{m+1} \|_2^2 + \| \tilde{p}^{m+1} \|_2^2 ) 
  + M^{(3)}   (  \| \tilde{n}^m \|_2^2 + \| \tilde{p}^m \|_2^2 )  
  + M^{(4)} h^8 
	\nonumber
  	\\
&&  + \| \tau_n^{m+1} \|_2^2 + \| \tau_p^{m+1} \|_2^2  ,  
  \label{convergence-5} 
\end{eqnarray}  
where
	\begin{align}
M^{(2)} &= 4 ( \tilde{C}_3 \tilde{C}_4 )^2 ( 1 + D^2 ) + 2 (M_n^{(0)} + M_p^{(0)} ) +1,
	\\
M^{(3)} &= 2 (   C^\star )^2 (\gamma_n^{(0)} )^{-1} +  2 ( D C^\star )^2 (\gamma_p^{(0)} )^{-1},
	\\
M^{(4)} &= 2 ( M_n^{(1)} + M_p^{(1)} ).
	\end{align}
Therefore, an application of discrete Gronwall inequality leads to the desired higher order convergence estimate
	\begin{equation}
  \| \tilde{n}^{m+1} \|_2 + \| \tilde{p}^{m+1} \|_2 + \Bigl( \dt   \sum_{k=1}^{m+1} 
   ( \| \nabla_h \tilde{n}^k \|_2^2 
+ \| \nabla_h \tilde{p}^k \|_2^2 ) \Bigr)^{1/2}  \le C ( \dt^3 + h^4 ) , 
	\label{convergence-6}
	\end{equation}  
based on the higher order truncation error accuracy, $\| \tau_n^{m+1} \|_2$, $\| \tau_p^{m+1} \|_2 \le C (\dt^3 + h^4)$. This completes the refined error estimate. 

\noindent
{\bf Recovery of the a-priori assumption~\eqref{a priori-1}} 

With the higher order error estimate~\eqref{convergence-6} at hand, we notice that the a-priori assumption in~\eqref{a priori-1} is satisfied at the next time step $t^{m+1}$:  
	\begin{equation} 
\| \tilde{n}^{m+1} \|_2,  \| \tilde{p}^{m+1} \|_2 \le \hat{C}_2 ( \dt^3 + h^4 ) \le \dt^\frac{11}{4} + h^\frac{11}{4} ,  
	\label{a priori-9}  
	\end{equation} 
provided $\dt$ and $h$ are sufficiently small. Therefore, an induction analysis could be applied. This finishes the higher order convergence analysis. 

As a result, the convergence estimate~\eqref{convergence-0} for the variable $(n, p)$ is a direct consequence of~\eqref{convergence-6}, combined with the definition~\eqref{consistency-1} of the  constructed approximate solution $(\check{\mathsf N}, \check{\mathsf P})$, as well as the projection estimate~\eqref{projection-est-0}. 

In terms of the convergence estimate for the electric potential variable $\phi$, we recall the definition for $\tilde{\phi}^k$ in~\eqref{error function-2} and observe that 
	\begin{equation}
\| \tilde{\phi}^m \|_{H_h^2} \le C \| \Delta_h \tilde{\phi^m} \|_2 \le \frac{C}{\varepsilon} \| \tilde{n}^m - \tilde{p}^m \|_2 \le \hat{C}_3 (\dt^3 + h^4) ,  
	\label{convergence-7-1} 
	\end{equation}
where $\hat{C}_3 =C \hat{C}_2$. Then
	\begin{equation}
  \| \tilde{\phi}^m - e_\phi^m \|_{H_h^2} 
  \le C \|  \Delta_h (\tilde{\phi}^m - e_\phi^m) \|_2 
  \le \hat{C}_4 (\dt + h^2) ,
	\label{convergence-7-3}    
	\end{equation} 
and
	\begin{align}
(-\Delta_h) (\tilde{\phi}^m - e_\phi^m) & = {\cal P}_N ( \dt {\mathsf P}_{\dt, 1} + \dt^2 {\mathsf P}_{\dt, 2} + h^2 {\mathsf P}_{h, 1} 
	\nonumber
	\\
& \quad - \dt {\mathsf N}_{\dt, 1} - \dt^2 {\mathsf N}_{\dt, 2} 
  - h^2 {\mathsf N}_{h, 1} ) + \tau_\phi^m ,   \label{convergence-7-2}   
	\end{align}
where the discrete elliptic regularity has been applied in~\eqref{convergence-7-1}, \eqref{convergence-7-3}, and the truncation error for $\phi$ is defined as $\tau_\phi^m = (-\Delta_h) \Phi_N -   (\check{\mathsf P}^m - \check{\mathsf N}^m)$. 
	
Finally, we arrive at 
	\begin{equation} 
\| e_\phi^m \|_{H_h^2} \le \| \tilde{\phi}^m \|_{H_h^2} +  \| \tilde{\phi}^m - e_\phi^m \|_{H_h^2} \le  \hat{C}_3 (\dt^3 + h^4)  +   \hat{C}_4 (\dt + h^2)  \le ( \hat{C}_4 + 1) (\dt + h^2) . 
	\label{convergence-7-4}    
	\end{equation} 
This completes the proof of Theorem~\ref{thm:convergence}. 

	\section{Numerical results}
	\label{sec:numerical results} 
%
To get numerical solutions, we need to solve the fully nonlinear scheme~\eqref{scheme-PNP-1} -- \eqref{scheme-PNP-chem pot-p} at each time step. We propose an iterative method as follows. First, the initial value for the nonlinear iteration is taken as $n^{m+1,0}:= n^m$, $p^{m+1,0}:= p^m$, and $\phi^{m+1,0}:=\phi^m$. Subsequently, given the $k$-th iterate numerical solution $n^{m+1,k}$, $p^{m+1,k}$, $\phi^{m+1,k}$, we obtain the first stage of the $(k+1)$-th iterate by solving 
\begin{equation} \label{solver-1} 
	\begin{aligned} 
n^{m+1,*} - \dt \nabla_h \cdot \left(\breve{\cal M}_n^m \nabla_h \left(\frac{n^{m+1,*}}{n^{m+1,k}}\right) \right)  & = n^{m} + \dt \nabla_h \cdot \left( \breve{\cal M}_n^m \nabla_h \left( \ln n^{m+1,k} - \phi^{m+1, k}\right)  \right), 
	\\
p^{m+1,*} - \dt \nabla_h \cdot \left(\breve{\cal M}_p^m \nabla_h \left(\frac{p^{m+1,*}}{p^{m+1,k}}\right) \right)   & = p^{m} + \dt \nabla_h \cdot \left( \breve{\cal M}_p^m \nabla_h \left( \ln p^{m+1,k} + \phi^{m+1, k}\right)  \right), 
	\\
-\Delta_h \phi^{m+1,*} & = p^{m+1,*} -n^{m+1,*}.
	\end{aligned} 
\end{equation} 
In addition, to make the nonlinear iteration smoother,  we then obtain $n^{m+1,k+1}$, $p^{m+1,k+1}$, and $\phi^{m+1,k+1}$ by
\begin{eqnarray}  \label{solver-3} 
\begin{aligned} 
\left(n^{m+1,k+1}, p^{m+1,k+1}, \phi^{m+1,k+1} \right) = & 
\omega_r  \left(n^{m+1,k}, p^{m+1,k}, \phi^{m+1,k} \right)
\\
   & 
  + (1- \omega_r)  \left(n^{m+1,*}, p^{m+1,*}, \phi^{m+1,*} \right), 
\end{aligned} 
\end{eqnarray} 
where $\omega_r \in (0,1)$ is a relaxation parameter. We notice that, two linear systems for $n$ and $p$, associated with ${\cal M}$-matrices, need to be solved in the $(k+1)$-th iteration algorithm~\eqref{solver-1}. In fact, \eqref{solver-1} could be viewed as a linearized Newton iteration for the proposed numerical scheme~\eqref{scheme-PNP-1} -- \eqref{scheme-PNP-chem pot-p}, at least in the $\ln n$ and $\ln p$ nonlinear parts. It is expected that, under a sufficient condition on the time step size $\dt$, such a linearized iteration algorithm guarantees positive concentrations at a discrete level in each iteration stage, and an iteration convergence to the proposed numerical scheme~\eqref{scheme-PNP-1} -- \eqref{scheme-PNP-chem pot-p} is also available. The detailed analysis will be left in the future works. 

In the following, we demonstrate the performance of the proposed numerical scheme in a two dimensional setting. With rescaling, the computational domain becomes $\Omega=(-1,1)^2$. Also, we take the parameters $z_0=1$, $n_0=p_0=c_0=0.1$ M, $L=13.6$ \AA, and $D_n=D_p=D_0$ with $D_0$ being the diffusion constant of sodium ions in water.


\subsection{Accuracy tests}
To test accuracy, we consider the following exact solution
\begin{equation}\label{ExtSoln}
\left\{
\begin{aligned}
&n=e^{-t}\sin(2\pi x)\cos(2\pi y )+2,\\
&p=e^{-t}\cos(2\pi x)\sin(2\pi y)+2,\\
&\psi=e^{-t}\sin(2\pi x)\sin(2\pi y),\\
\end{aligned}
\right.
\end{equation}
to the PNP equations with source terms:
		\begin{eqnarray} 
\partial_t n &=&   \nabla\cdot\left( \nabla n - n \nabla \phi  \right) +f_n  ,   
	\label{RescaledPNP-1} 
	\\
\partial_t p &=&  \nabla\cdot\left( \nabla p +  p \nabla \phi  \right)  +f_p ,   
	\label{RescaledPNP-2}  
	\\
 -  \Delta \phi &=&  p-n +\rho^f .
 	\label{RescaledPNP-3} 
	\end{eqnarray}
Here the source terms $f_n$, $f_p$, and $\rho^f$, and the initial conditions are obtained with the known exact solution. 


\begin{table}
\begin{center}
\begin{tabular}{ccccccc}
\hline  \hline
 $h$ & $\ell^\infty$ error in $p$ & Order & $\ell^\infty$ error in $n$ & Order & $\ell^\infty$ error in $\psi$ & Order\\
 \hline
 0.1 & 1.898E-2 & - & 1.898E-2 & - & 1.200E-1 & -\\
 0.05 & 4.864E-3 & 1.96 & 4.864E-3 & 1.96 & 3.001E-2 & 2.00\\ 
 0.025 & 1.231E-3 & 1.98 & 1.231E-3 & 1.98 & 7.524E-3 & 2.00\\
 0.0125 & 3.093E-4 & 1.99 & 3.093E-4 & 1.99 &  1.882E-3 & 2.00\\
 \hline  \hline
\end{tabular}
\caption{The $\ell^\infty$ error and convergence order for the numerical solutions of $p$, $n$, and $\psi$ with $\Delta t=h^2$.}
\label{t:convergence}
\end{center}
\end{table}
To verify the accuracy of the proposed scheme \reff{scheme-PNP-1} to \reff{scheme-PNP-chem pot-p}, we perform numerical tests using various mesh resolution with $\Delta t=h^2$. Note that such a mesh ratio is chosen for the purpose of accuracy tests rather than the stability concern.  As shown in Table \ref{t:convergence},  the $\ell^{\infty}$ error for numerical solutions of $p, n$, and $\phi$ at time $T=0.1$ decreases robustly as the mesh refines. The convergence order, as expected, is about two for both the concentrations and electrostatic potential.
\subsection{Properties tests}
We also conduct numerical simulations to test the performance of the proposed scheme in preserving physical properties at discrete level. The numerical schemes are applied to solve the rescaled PNP equations~\reff{RescaledPNP-1} to~\reff{RescaledPNP-3} without sources terms in the Nernst--Planck equations, but with a fixed charge distribution given by
	\[
\rho^f(x,y)= e^{-100\left[(x+\frac{1}{2})^2+(y+\frac{1}{2})^2\right]} -e^{-100\left[(x+\frac{1}{2})^2+(y-\frac{1}{2})^2\right]}-e^{-100\left[(x-\frac{1}{2})^2+(y+\frac{1}{2})^2\right]} +e^{-100\left[(x-\frac{1}{2})^2+(y-\frac{1}{2})^2\right]}.
	\]
The initial data for concentrations are given by
	\[
p(0,x,y)=0.1 \quad \mbox{and}\quad  n(0,x,y)=0.1.
	\]
\begin{figure}[ht]
\centering
\includegraphics[scale=.44]{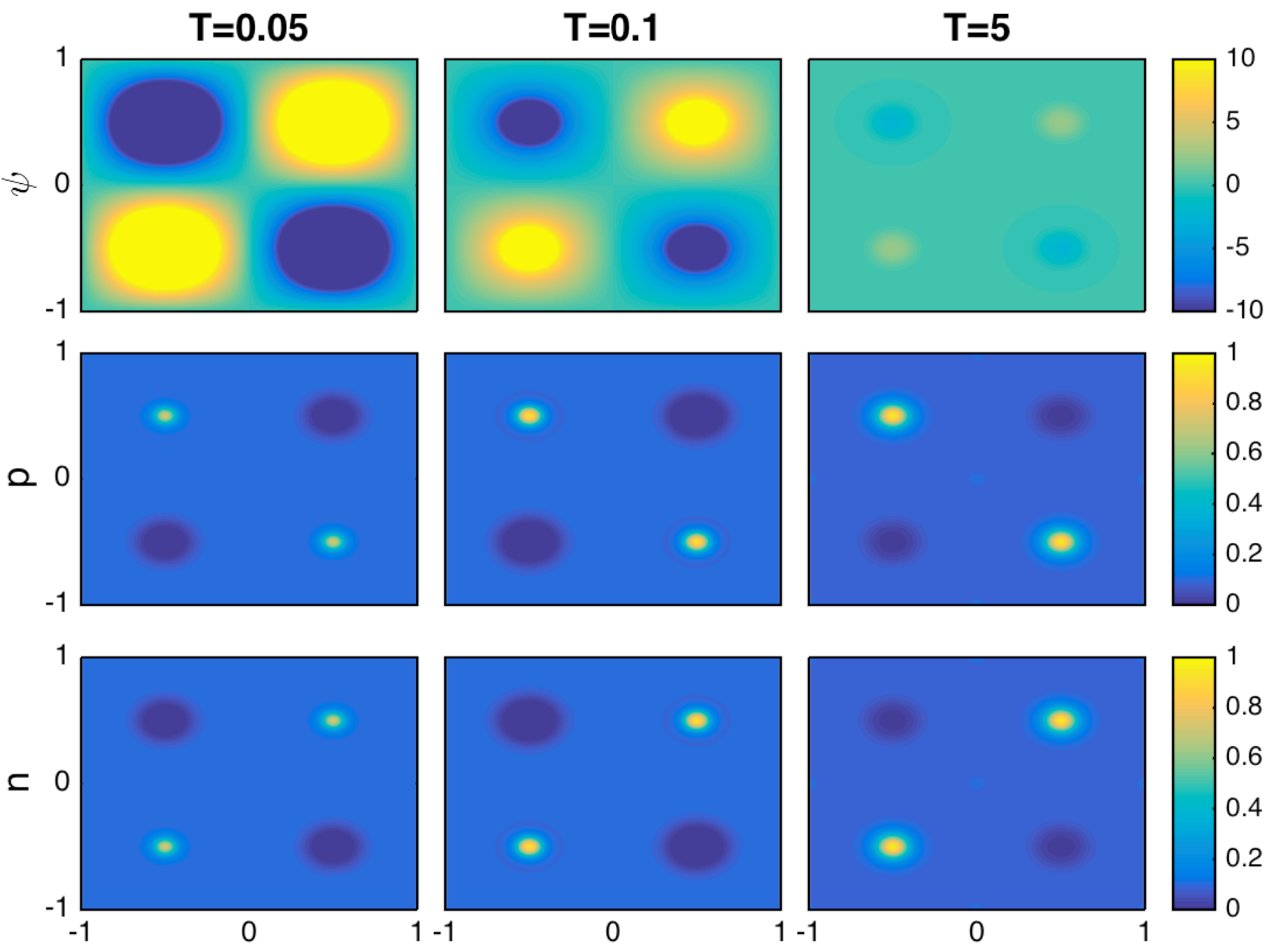}
\includegraphics[scale=.41]{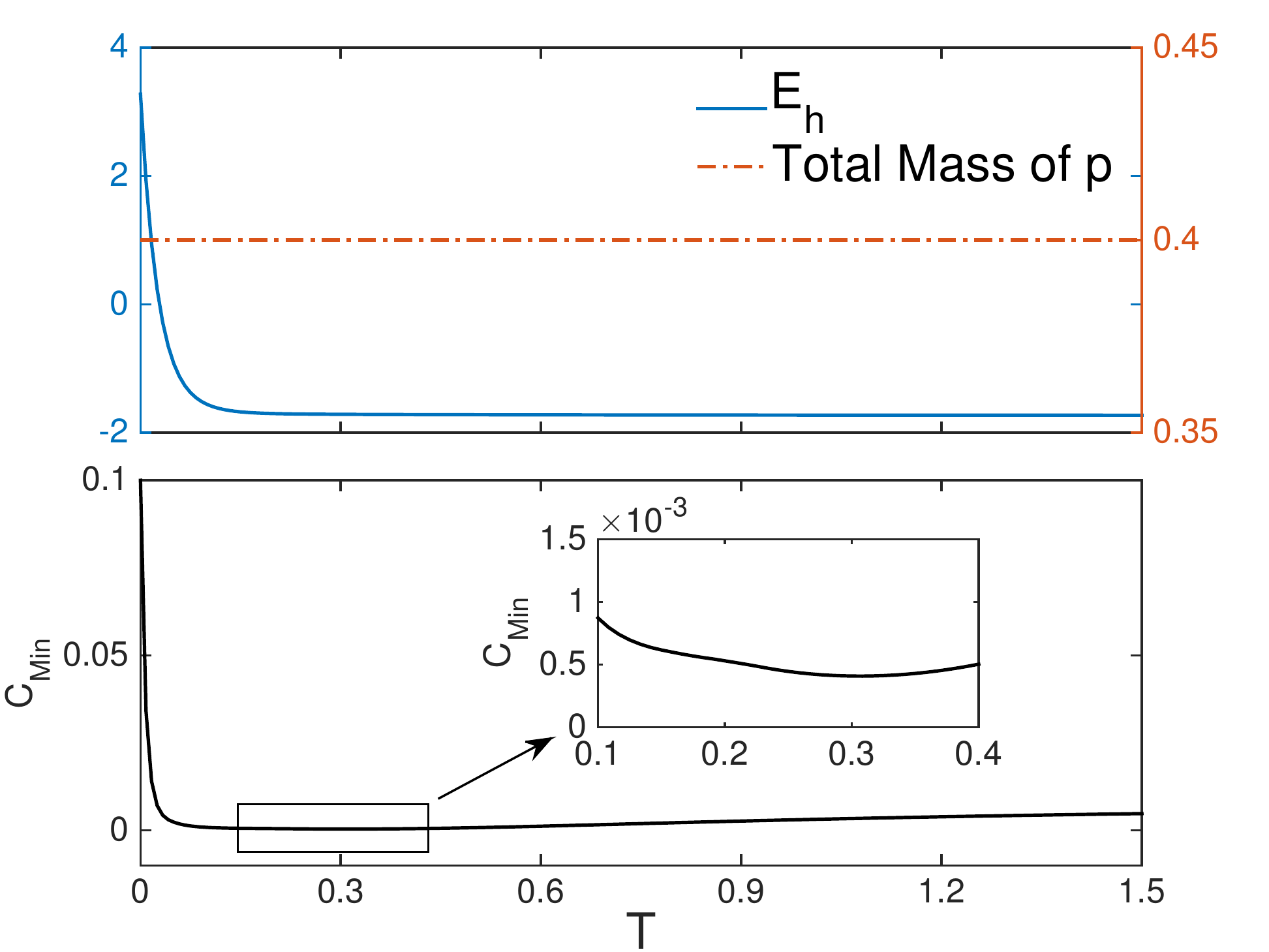}
\caption{Left: The snapshots of $\psi$, $p$, and $n$ at time $T=0.05$, $T=0.1$, and $T=5$. Right: The evolution of discrete energy $E_h$, total mass of $p$, and the minimum concentration $C_{\rm Min}$. }
\label{f:PsiNP}
\centering
\end{figure}


Figure~\ref{f:PsiNP} displays snapshots of the electrostatic potential and concentrations at time $T=0.05$, $T=0.1$, and $T=5$. One observes that the concentrations of cations and anions develop peaks and valleys due to electrostatic interactions, and that the electrostatic potential initially induced by the fixed charges gets screened quickly by attracted mobile ions carrying opposite charges, as time evolves.  At $T=5$, the system nearly reaches equilibrium.  

By periodic boundary conditions, the total mass of concentrations is conserved in time evolution. This is verified in the right panel of the  Fig.~\ref{f:PsiNP}, in which the total mass of the cations converses perfectly. In addition, the discrete energy $E_h$ decreases monotonically, being consistent with our analysis; cf.~Theorem~\ref{PNP-energy stability}. Of interest is the evolution of the minimum concentration that is defined by $C_{\rm Min}:={\rm Min}\{{\rm Min}_{i,j,k} n_{i,j,k}^m, {\rm Min}_{i,j,k} p_{i,j,k}^m \}$. The evolution of $C_{\rm Min}$, together with the inset plot, demonstrates that the numerical solution of concentration remain positive all the time.  In summary, our numerical tests further confirm that the proposed numerical scheme respects mass conservation, energy dissipation, and positivity at discrete level. 

\section{Concluding remarks}  \label{sec:conclusion} 
A finite difference numerical scheme is proposed and analyzed for the Poisson-Nernst-Planck (PNP) system. The Energetic Variational Approach (EnVarA) is taken, so that the PNP system could be reformulated as a non-constant mobility $H^{-1}$ gradient flow, with singular logarithmic energy potentials involved. In the proposed numerical algorithm, the mobility function is explicitly treated to ensure the unique solvability, while both the logarithmic and the electric potential diffusion terms are treated implicitly, because of their convex natures. The positivity-preserving property for both $n$ and $p$ are theoretically established, which is based on the subtle fact that, the singular nature of the logarithmic term around the value of $0$ prevents the numerical solution reaching the singular value. As a result, the numerical scheme is always well-defined. The energy stability of the numerical scheme comes from the convex nature of the energy functional in terms of $n$ and $p$, combined with their positivity property. In addition, an optimal rate convergence analysis is provided in this work. To overcome a well-known difficulty associated with the non-constant mobility, many highly non-standard estimates have to be involved, due to the nonlinear parabolic coefficients. The higher order asymptotic expansion, up to third order temporal accuracy and fourth order spatial accuracy, has to be performed with a careful linearization technique. Such a higher order asymptotic expansion enable one to obtain a rough error estimate, so that to the $\ell^\infty$ bound for $n$ and $p$ could be derived. This $\ell^\infty$ estimate yields the upper and lower bounds of the two variables, and these bounds have played a crucial role in the subsequent analysis. Finally, the refined error estimate are carried out to accomplish the desired convergence result. It the first work to combine three theoretical properties for any numerical scheme to the PNP system: unique solvability/positivity-preserving, energy stability and optimal rate convergence analysis. A few numerical results are also presented in this article, which demonstrates the robustness of the proposed numerical scheme.

	\section*{Acknowledgements} 
This work is supported in part by the National Science Foundation (USA) grants NSF DMS-1759535, NSF DMS-1759536 (C.~Liu), NSF DMS-2012669 
(C.~Wang),  NSF DMS-1719854, DMS-2012634 
(S.~Wise), National Natural Science Foundation of China 11971342 (X.~Yue), 21773165, Young Elite Scientist Sponsorship Program by Jiangsu Association for Science and Technology, Natural Science Foundation of Jiangsu Province, China,  and National Key R\&D Program of China 2018YFB0204404 (S. Zhou).

	\bibliographystyle{plain}
	\bibliography{pnp.bib}

\end{document}